\documentclass{EGAstyle}
 
\usepackage{indentfirst} 
\usepackage{color} 
\begin{document}

\begin{center}
    {\Huge Nonrelativistic Limit  of  Normalized Solutions to a class of nonlinear Dirac equations}
\end{center}
\begin{center}
 {~~Pan Chen$^{b,c}$, Yanheng Ding$^{b,d}$, Qi Guo$^{a}\footnote[1]{Corresponding author: qguo@ruc.edu.cn}$ and Hua-Yang Wang$^{b,c}$}\\
\small School of Mathematics, Renmin University of China, Beijing, 100872, P.R. China$^{a}$\\
 \small Academy of Mathematics and System Science, Chinese Academy of Sciences, Beijing, 100190, P.R. China$^{b}$\\
 \small University of Chinese Academy of Sciences, Beijing, 100049, P.R. China$^{c}$\\
\small School of Mathematics, Jilin University, Changchun, 130012, P.R. China$^{d}$
\end{center} 
 
{ \bf Abstract  }:
{\small In this paper, we investigate the nonrelativistic limit of normalized solutions to a nonlinear Dirac equation as given below:
\begin{equation*}
  \begin{cases}
&-i c\sum\limits_{k=1}^3\alpha_k\partial_k u  +mc^2 \beta {u}- \Gamma * (K |{u}|^\kappa) K|{u}|^{\kappa-2}{u}- P |{u}|^{s-2}{u}=\omega {u}, \\
&\displaystyle\int_{\mathbb{R}^3}\vert u \vert^2 dx =1.
  \end{cases}
\end{equation*}
Here, $c>0$ represents the speed of light, $m > 0$ is the mass of the Dirac particle, $\omega\in\mathbb{R}$ emerges as an indeterminate Lagrange multiplier, $\Gamma$, $K$, $P$ are real-valued function defined on $\mathbb{R}^3$, also known as potential functions.
Our research first confirms the presence of normalized solutions to the Dirac equation under high-speed light conditions. We then illustrate that these solutions progress to become the ground states of a system of nonlinear Schr\"odinger equations with a normalized constraint, exhibiting uniform boundedness and exponential decay irrespective of the light speed. Our results form the first discussion on  nonrelativistic limit of normalized solutions to nonlinear Dirac equations. This not only aids in the study of normalized solutions of the nonlinear Schr\"odinger equations, but also physically explains that the normalized ground states of high-speed particles and low-speed motion particles are consistent.} \\


\noindent {\bf AMS} Subject Classification: \small 35Q40  35J50 49J35.
 
\section{Introduction}
%In this paper, we consider the Dirac equation with both local and nonlocal nonlinear terms. A model question is the nonlinear Dirac-Coulomb system:
%\begin{equation}\label{eq_Dirac_g}
%   \begin{cases}
%       i\hbar c\boldsymbol{\alpha}\cdot \nabla  u +mc^2\beta  u +\omega u-K(x)e\phi u -\frac{1}{2}P(x)  \vert  u  \vert^{s-2}  u =0,\\
%        -c\Delta \phi = 4\pi K(x)e\vert u \vert^2.
%    \end{cases}
%\end{equation}

In this paper, we study the following stationary nonlinear Dirac equation:
\begin{equation}\tag{NDE}\label{eq_Dirac_g}
-i  c\sum\limits_{k=1}^3\alpha_k\partial_k u  +mc^2 \beta {u}- \omega {u}=\Gamma * (K |{u}|^\kappa) K|{u}|^{\kappa-2}{u}+ P |{u}|^{s-2}{u},
\end{equation}
where $ u : \mathbb{R}^3 \rightarrow \mathbb{C}^4$, $\partial_k=\frac{\partial}{\partial x_k}$,  $*$ stands for the convolution, $\omega\in\mathbb{R}$ and $\alpha_1, \alpha_2,\alpha_3,\beta$ are $4\times 4$ Pauli-Dirac matrices:
$$
\alpha_k=\left(\begin{array}{cc}
0 & \sigma_k \\
\sigma_k & 0
\end{array}\right), \quad \beta=\left(\begin{array}{cc}
I & 0 \\
0 & -I
\end{array}\right),
$$
with
$$
\sigma_1=\left(\begin{array}{cc}
0 & 1 \\
1 & 0
\end{array}\right), \quad\sigma_2=\left(\begin{array}{cc}
0 & -i \\
i & 0
\end{array}\right), \quad\sigma_3=\left(\begin{array}{cc}
1 & 0 \\
0 & -1
\end{array}\right) .
$$
We aim to search for normalized (weak) solutions of \eqref{eq_Dirac_g} which is defined in  section 3 and study the limit properties of these normalized solutions. 
 
%It is not difficult to check that $\beta$ and $\alpha_k$ satisfy the following anticommutation relations
%$$
%\left\{\begin{array}{l}
%\alpha_k \alpha_l+\alpha_l \alpha_k=2 \delta_{k l} I_4 \\
%\alpha_k \beta+\beta \alpha_k=0 \\
%\beta^2=I_4
%\end{array}\right.
%$$

\subsection{Motivations}
Nonlinear versions of the Dirac equation have been extensively studied since the invention of the original equation to incorporate additional effects and interactions, see \cite{WOS:A1988P347300005,WOS:000267267900023,WOS:000276973700004}. It is noteworthy that in Schr\"odinger equations, nonlinearities occur as approximations in optics and condensed matter \cite{WOS:000230275500030}. At the quantum mechanical level, researchers hope to detect quantum nonlinearities at high energy or very short distances, neutrino oscillations may be linked to quantum nonlinearities \cite{WOS:000276973700004}. 
Nonlinear Dirac equations serve as influential theories in various fields, including atomic, nuclear, and gravitational physics \cite{Finkelstein1}. These equations depict the behavior of fermions subjected to intense gravitational forces, providing invaluable insights into the behavior of matter under extreme conditions. Additionally, nonlinear Dirac equations find applications in condensed matter physics, where they effectively describe the actions of electrons in specific materials\cite{WOS:000267267900023}.
One of its general forms is
\begin{align*}\label{nonlinear}
i  c\hbar \gamma^\mu\partial_\mu\psi-mc^2  \psi- G_{\bar{\psi}}+\partial_\mu ( G_{\partial_\mu \bar{\psi}})=0,
\end{align*}
where $\psi:\mathbb{R}^{3+1}\rightarrow \mathbb{C}^4$ represents the wave function of the state of Dirac particles, such as electron, $\bar{\psi}$ is the Dirac adjoint of $\psi$, $m>0$ is the mass of the particle, $c$ is the speed of the light,
$\hbar$ is the Planck constant, $G$ is a function of the wave function $\psi$, its adjoint and their derivatives.
If $G$ satisfies
\[
\overline{G_\psi}-\partial_\mu (\overline{G_{\partial \mu\psi}})=G_{\bar{\psi}}-\partial_\mu ( G_{\partial_\mu \bar{\psi}}).
\]
Then the nonlinear Dirac equation can be obtained from the action with the Lagrangian density given by
\[\mathcal{L}=ic\hbar \bar{\psi} \gamma^\mu \partial_\mu \psi-mc^2\bar{\psi}\psi- G(\psi,\bar{\psi},\partial_\mu \psi,\partial_\mu\bar{\psi}).\]
To construct nonlinearities of Dirac equations, it is necessary to take into account the fundamental properties of linear theory, such as Poincar\'e invariance, current conservation, and $\mathcal{PCT}$ invariance. However, by imposing additional conditions on the function $G$, it is possible to preserve some of these properties, as described in \cite{WOS:000267267900023}.  These conditions are crucial for constructing physically meaningful nonlinear Dirac equations, and their study can lead to a better understanding of the behavior of fermions in complex systems. 
The linear Dirac equation exhibits invariance under wavefunction rescaling. This feature in quantum mechanics permits the normalization of solutions, which is not only advantageous but also occasionally necessitated for the interpretation of experimental observations \cite{MR516997, WOS:A1989AU60000005}. Due to the fact that nonlinear terms typically do not possess scaling invariance, we need to introduce the additional constraint, that is $$\displaystyle \int_{\mathbb{R}^3} |\psi(t,x)|^2 dx=1,$$ to find normalized solutions for the nonlinear Dirac equation.

\medskip
We search for solitary solutions or solitons of the nonlinear Dirac equations, which take the form $$\psi(t,x)=e^{i\omega t}u(x).$$ Here $\omega$ is known as the frequency of the wave function.  We simplify the model by setting $\hbar=1$ and using a function $F:\mathbb{R}^3\times\mathbb{C}^4\rightarrow \mathbb{R}$ instead of $G$. Using this simplified model, we only need to search for solutions to the following stationary nonlinear Dirac equations:
 
\begin{equation*}\label{Dirac}
-i  c\sum\limits_{k=1}^3\alpha_k\partial_k u  +mc^2 \beta {u}-\omega {u}=F_u(x,u).
\end{equation*}
The Dirac-Slater model is a widely used quantum mechanical model for determining the binding energies and eigenfunctions of molecules. It has a nonlinear term arising from the potential's dependence on the electron distribution, which makes it a valuable tool for understanding the behavior of electrons within molecular systems. Relativistic effects become significant for atomic systems with a nuclear charge greater than 30, and the Dirac-Slater model has been applied in atomic research and tested on simple molecules. It is also used to describe atomic levels and energy band structure in crystals. The uranyl ion $UO_2^{2+}$ has garnered significant attention in both theoretical and experimental studies due to its remarkable stability in solutions and its presence as a distinctive complex in solid compounds \cite{Ellis}. In terms of the theory of Dirac-Slater model, the scalar potential $V$ is taken to be a sum of Coulomb term $V_c(r)$ and exchange term $V_x(r)$. That is  
\[V(r)=V_c(r)+V_x(r),\]
and the Dirac-Slater model is defined by a free-electron exchange approximation given by 
\[V_x(r)=-3C_{KS} \left(\frac{3\rho(r)}{4\pi}\right)^{1/3}.\]
Here $\rho(r)$ is the total charge density and $C_{KS} $ is the Kohn-Sham exchange coefficient.
The relativistic correction to the exchange potential in the Dirac-Slater equation was developed in \cite{Ellis}. The unpolarized relativistic exchange potential in the Dirac-Kohn-Sham-Slater model can be expressed approximately as 
\[V_x^{rel}(r)=-3 \left(\frac{3\rho(r)}{4\pi}\right)^{1/3}\left ( 1-\frac{\lambda}{9}-\frac{7\lambda^2}{180}\right),\]
where $\lambda =C_{KS} ^2k^2/(C_{KS} ^2 k^2+1)$, the Fermi momentum $k=(6\pi^2\rho)^{1/3}$, see \cite{diracslater3}. In terms of the wave function $\psi$, the total electron density can be expressed as 
\[\rho_\sigma =\sum_{\sigma} N_\sigma |\psi_{i,\sigma}|^2,\]
where $\sigma=\uparrow$ or $\downarrow$, $N_\sigma$ is the occupation numbers of spin up ($\uparrow$) or down ($\downarrow$) electrons.  In atomic units, the one-particle Dirac-Slater equation becomes 
\[ -i  c\sum\limits_{k=1}^3\alpha_k\partial_k u +mc^2\beta\psi +V_c\psi - C_{ex} |\psi|^{2/3} \psi=\omega \psi,\]
where $C_{ex}=3C_{KS}\left(\frac{3N}{4\pi}\right)^{1/3}$. The idea of nonlinear term in the Dirac-Slater model comes from two parts, one is the influence of the particle's own charge, and the other is the nonlinear effect from different molecules. The charge's influence on the particle's motion is a non-local nonlinear effect, while the perturbation of external fields is a local nonlinear effect. Now we delve into the nonlinear extension of this model, which serves as the primary focus of our paper. It has the following form:
 \[F(x,u)=\frac{1}{2\kappa}\Gamma * (K |{u}|^\kappa)(x) K(x)|{u}|^{\kappa}+ \frac{1}{s}P(x) |{u}|^{s}.\]

There are two main approaches to study solutions of the equation \eqref{eq_Dirac_g}. The first is known as the \emph{Fixed Frequency Problem}, which involves fixing $\omega \in (-mc^2,mc^2)$ and searching for solutions $u$ to the equation.
Esteban and S\'{e}r\'{e} \cite{MR1344729} were the first to introduce variational methods to study stationary solutions  of the nonlinear Dirac equation.
Later, Esteban, Georgiev, and S\'{e}r\'{e} extended this method to Dirac-Maxwell and  Dirac-Klein-Gordon systems in \cite{MR1386737}.
Following the works \cite{MR1386737} and \cite{MR1344729}, a large number of papers have been devoted to the study of the existence and multiplicity of stationary solutions of the Dirac equation with both local and nonlocal nonlinearities under
various assumptions using variational methods, for example, see \cite{MR2232435, MR4256273, Pan_2022, MR2434900, MR4327893} and references therein.
It is worth pointing out that the action functional of the Dirac equation \eqref{eq_Dirac_g} is strongly indefinite and one can apply variational methods in \cite{MR2389415} to study the existence and multiplicity of solutions.
On the other hand, a natural way is to fix the $L^2$-norm of $u$ and solve \eqref{eq_Dirac_g} with an unknown number $\omega\in \mathbb{R}$, which arises as a Lagrange multiplier.
We usually call it the \emph{Normalized Solutions Problem}.
 Buffoni,  Esteban and S\'{e}r\'{e}  \cite{MR2219842} dealt with the existence of normalized solutions of strongly indefinite semilinear elliptic equations
via an unconstrained penalization, a variational Lyapunov-Schmidt reduction, and a mountain-pass theorem.
Recently, Ding, Yu and Zhao considered the existence of normalized solutions to a nonlinear Dirac equation in \cite{MR4523529}.
Actually, our idea mainly comes from Coti Zelati and  Nolasco, they studied the existence of a ground state solution to the Dirac-Maxwell system in \cite{MR3953461} and also studied normalized solutions for the Klein–Gordon–Dirac system in \cite{MR4628735}.
 Nolasco  \cite{MR4327893} considered the existence of normalized solutions to a Dirac-Maxwell equation.

\medskip

The nonrelativistic limit arises when the velocity of a particle is significantly lower than the speed of light. In this limit, the influence of special relativity becomes negligible, and classical mechanics suffices to explain the particle's motion. By considering the nonrelativistic limit, we assume that the particle's velocity is much smaller than the speed of light. Consequently, the laws of special relativity align with Newtonian physics, and relativistic effects can be disregarded. This simplification allows for a more straightforward and intuitive analysis of the particle's behavior compared to a full relativistic treatment. There are many works studied the nonrelativistic limit of the Cauchy problem for partial differential equations arising from relativistic quantum mechanics, see \cite{MR2134957,MR1895710,MR1993419,MR1151464}.
%S. Machihara, K. Nakanishi and T. Ozawa \cite{MR1895710} studied the nonrelativistic limit of the Cauchy problem for the nonlinear Klein-Gordon equation and deduced that any finite energy solution converges to a nonlinear Schr\"{o}dinger equation in the energy space.
Some authors also studied the nonrelativistic limit of stationary solutions to nonlinear Dirac equations. For example, Esteban and S\'{e}r\'{e} \cite{MR1869528} described a limiting process that shows how solutions of Dirac-Fock equations converge toward the corresponding solutions of Hartree-Fock equations when the speed of light tends to infinity. This work is based on the existence results of normalized solutions of the Dirac-Fock system in \cite{MR1869528}.
Borrelli, Carlone and Tentarelli \cite{MR3934110, WOS:000613291600011} studied the existence and multiplicity of the bound states to nonlinear Dirac equations on non-compact metric graphs with Kirchhoff-type conditions and they proved that these bound states converge to the bound states of the nonlinear Schr\"odinger equation in the nonrelativistic limit.  
Recently,  Dong, Ding and Guo \cite{MR4280523, dongdingguo} studied the nonrelativistic limit and some properties of solutions for the stationary Dirac equation with a local nonlinearity.

\medskip
Motivated by the above papers, the purpose of this paper is to study the nonrelativistic limit of normalized solutions of the Dirac equation with both local and nonlocal nonlinearities.
Compared with the {\it Fixed Frequency Problem}, it seems easier and more reasonable to consider the nonrelativistic limit of the Dirac equation from the viewpoint of normalized solutions.
Indeed, we shall describe the relation between the frequency $\omega$ and the light speed $c$, which is more natural than that arising as a basic assumption in \cite{MR3934110,MR4280523}.

%We point out that nonlinear terms in the above works are purely either local or nonlocal, which means the normalized solutions problem for the Dirac equation with mixed local and nonlocal nonlinearities is open. Our first aim is to study the normalized solutions problem for a more or less general Dirac equation  \eqref{eq_Dirac_g} with both local and nonlocal nonlinearities.

\subsection{Main Resutls}
We now consider the following constraint Dirac equation
\begin{align}\label{eq:Dirac}\tag*{$(NDE)_c$}
  \begin{cases}
&-i c\sum\limits_{k=1}^3\alpha_k\partial_k u  +mc^2 \beta {u}- \Gamma * (K |{u}|^\kappa) K|{u}|^{\kappa-2}{u}- P |{u}|^{s-2}{u}=\omega {u}, \\
&\displaystyle\int_{\mathbb{R}^3}\vert u \vert^2 dx =1.
  \end{cases}
\end{align}
To state our main theorems, we need the following assumptions on the nonlinearities:
\medskip

 ($K_1$) $K\in \mathcal{C}^1(\mathbb{R}^3,(0,+\infty))$ and $\lim\limits_{|x|\to \infty}K(x)=0$.

($P_1$) $P \in \mathcal{C}^1(\mathbb{R}^3,(0,+\infty))$ and $\lim\limits_{|x|\to \infty}P(x) =0$.

($P_2$) There exist a constant $C >0$, a number $\mu \in \left(0, \frac{10-3s}{2}\right)$ such that for small $\varepsilon>0$, and all $x\in \mathbb{R}^3$, it holds that $P ( x)\geq  C \varepsilon^\mu P(\varepsilon x)$.

$(\Gamma_1)$ $\Gamma \in L_w^{6/(14-6\kappa)}(\mathbb{R}^3)\cap \mathcal{C}(\mathbb{R}^3\backslash\{0\},(0,+\infty)).$

\medskip
Here, $L_{w}^p(\mathbb{R}^3)$
is the weak $L^p$ space with $1<p<+\infty$ and norm $$\|u\|_{p,w}:=\sup\limits_{t>0} \left(t\mu^{1/p}_L\{x\in \mathbb{R}^3: |u(x)|>t\}\right),$$ 
where $\mu_L$ is the usual Lebesgue measure on $\mathbb{R}^3$. 
We give some examples which satisfy our assumptions:
\medskip

(i) $K(x)=e^{-a|x|}$, where $a> 0$. 

(ii) $P(x)=\frac{1}{1+|x|^\mu}$, where $\mu$ is given in assumption $(P_2)$.

(iii) $\Gamma(x)=\frac{1}{|x|^\tau  }$, where $\tau\in(0, 7-3\kappa)$.

\medskip

There are many other examples of functions defined on the real number line that vanish at infinity, i.e. Gaussian function: $K(x) = e^{-|x|^2}$, arctangent function: $K(x) = \frac{\arctan(|x|)}{|x|}$. 
These functions are commonly used due to their interesting and useful properties. Now, it is time to describe our main results.
\begin{theorem}\label{them:1.1}
  Set $\kappa\in [2,7/3), s\in (2,8/3]$. If assumptions $(K_1), (P_1), (P_2)$ and $(\Gamma_1)$ hold, then for a given $c>0$ large enough,  there exists $\omega_c\in (0,mc^2)$ and a function $u_{c}\in H^{1}(\mathbb{R}^3,\mathbb{C}^4)$, such that
  $(\omega_c, u_{c})$ is a normalized solution of  \ref{eq:Dirac}.
  In addition, we have
  $$-\infty < \liminf\limits_{c\to \infty} (\omega_c-mc^2)\leqslant \limsup\limits_{c\to \infty} (\omega_c-mc^2)< 0.$$
\end{theorem}

The definition of normalized solution of \ref{eq:Dirac} can be found in section 3. As a consequence of this theorem, there is a subsequence $\{c_n\}$, the corresponding normalized solutions are denoted by $(\omega_{c_n}, u_{c_n})$, and there exists a positive constant $\nu>0$ satisfying $$\lim\limits_{n\rightarrow \infty} (mc_n^2-\omega_{c_n})=\frac{\nu}{m}.$$
Naturally, we would ask whether the solutions of the above system will approach the solutions of the nonlinear Schr\"odinger equation in the sense of the nonrelativistic limit. The following theorem tells us that normalized solutions of the nonlinear Dirac equation converge to normalized solutions of the nonlinear Schr\"odinger equation in the sense of the nonrelativistic limit.

\begin{theorem}\label{them:1.2}
  Under assumptions of Theorem \ref{them:1.1},
there is a positive constant $\nu>0$, and a sequence $(\omega_{c_n}, u_{c_n}=(f_{c_n},g_{c_n})^T)$ which is a solution of $(NDE)_{c_n}$, such that 
$$f_{c_n} \to h, \quad g_{c_n}\to 0 \quad \text{in}\quad H^1(\mathbb{R}^3,\mathbb{C}^2),$$
  where $  h:\mathbb{R}^3\rightarrow \mathbb{C}^2$ is a  ground state solution of the following stationary nonlinear Schr\"odinger equations:
  \begin{equation}\label{NSE}\tag{$NSE$}
    \begin{cases}
     & -\Delta h +\nu h=  2mP |h|^{s-2}h + 2m \Gamma*(K |h|^\kappa)K|h|^{\kappa-2}h, \\
  &   \displaystyle\int_{\mathbb{R}^3} |h |^2dx=1.
    \end{cases}
  \end{equation}
\end{theorem}

The definition of ground state solution of \eqref{NSE} can be found in section 4.
From Theorem \ref{them:1.2}, it can be seen that the method of nonrelativistic limits can also be used to find normalized solutions of nonlinear Schr\"odinger equations. 
\begin{theorem}\label{them:1.3}
Under assumptions of Theorem \ref{them:1.1}, it holds that $\left\{u_c\right\}$ is bounded in $L^{\infty}\left(\mathbb{R}^3,\mathbb{C}^4\right)$.
Moreover, for any $0<\delta <\sqrt{\nu}$, there exists $c_0,C>0$ such that for any $c>c_0$, we have
$$
\left|u_c(x)\right| \leq C e^{-\delta|x|}, \quad \forall x \in \mathbb{R}^3.
$$
\end{theorem}
The findings of this paper reveal the nonrelativistic limit of the normalized solution for the nonlinear Dirac equation. By assuming a sufficiently large speed of light, it becomes easier to establish the existence of normalized solutions. From a physical perspective, this signifies the presence of normalized solitary solutions when particle motion is relatively slow.
Similar to the nonrelativistic limit result, the normalized solution of the nonlinear Dirac equation approaches the normalized solution of the nonlinear Schr\"odinger equation. Additionally, the normalized assumption simplifies the process of estimating the energy of the solution, leading to the conclusion that the approximated normalized solution is the ground state.
These results are fascinating both mathematically and physically. Mathematically, they allow us to confirm the existence of normalized solutions for the nonlinear Schr\"odinger equation through the nonrelativistic limit and directly determine that it is the ground state. Physically, this finding is reasonable, as it suggests that the energy levels of Dirac particles in low-speed states will remain unchanged.
\medskip

This paper is organized as follows.
In section 2, we recall some useful inequalities and lemmas.
Section 3 is devoted to a discussion of the existence of the normalized solution of Dirac equations.
Then we study the nonrelativistic limit and some properties of Dirac equations under normalized constraint in section 4.
We give proof of some useful lemmas in the Appendix.

\medskip
\textbf{Notations. }
Throughout this paper, we make use of the following notations.
\begin{itemize}
    \item For any $R>0$, $B_R$ denotes the ball of radius $R$ centered at original point;
    \item $\|\cdot\|_{L^q}$ denotes the usual norm of the space $L^q\left(\mathbb{R}^3, \mathbb{C}^4\right)$ \text{or} $L^q\left(\mathbb{R}^3, \mathbb{C}^2\right)$;
    \item $\|\cdot\|_{H^1}$ denotes the usual norm of the space $H^1\left(\mathbb{R}^3, \mathbb{C}^4\right)$ \text{or} $H^1\left(\mathbb{R}^3, \mathbb{C}^2\right)$;
\item $C$ is some positive constant that may change from line to line;
\item $(\cdot, \cdot)_{\mathbb{C}^4}$ denotes the usual complex inner product in $\mathbb{C}^4$;
\item  $a\lesssim  b$ means that $a \leq   C b$.
\item $\alpha\cdot\nabla$ (or $\sigma\cdot\nabla$) means that $\sum\limits_{k=1}^3 \alpha_k\partial_k$ (or $\sum\limits_{k=1}^3 \sigma_k\partial_k$).
\item $\Re(\Im)$ stands for the real part (image part) of a  complex valued function.
    \end{itemize}

\section{Preliminary Results}

We denote $\mathcal{F}(u)$ or $\hat{u}$ the Fourier transform of $u$, which is defined by
$$
\hat{u}(\xi)=\frac{1}{(2 \pi)^{3/2}} \int_{\mathbb{R}^3} e^{-i \xi \cdot x} u(x) d x .
$$
For $u, v\in H^{1/2}(\mathbb{R}^3,\mathbb{C}^4)$, the inner product in $H^{1/2}\left(\mathbb{R}^3, \mathbb{C}^4\right)$  is defined by
$$
(u, v)_{H^{1/2}}:=\Re \int_{\mathbb{R}^3} \sqrt{|\xi|^2+1}(\hat{u}(\xi), \hat{v}(\xi))_{\mathbb{C}^4} d \xi.
 $$
For convenience, the free Dirac operator is denoted by
$$
\mathscr{D}_c:= -i c\sum\limits_{k=1}^3\alpha_k\partial_k + mc^2 \beta.
$$
It is clear that $\mathscr{D}_c$ is self-adjoint on $L^2(\mathbb{R}^3,\mathbb{C}^4)$ with domain $\mathcal{D}(\mathscr{D}_c)\cong H^{1}(\mathbb{R}^3,\mathbb{C}^4)$ for any $c>0$, and we have
$$\sigma(\mathscr{D}_c-\omega)=(-\infty,-mc^2-\omega]\cup[mc^2-\omega,+\infty),$$  where $\sigma(\cdot)$ is the spectrum
of the linear operator.
Therefore, the Hilbert space $L^2(\mathbb{R}^3, \mathbb{C}^4)$ possesses the following orthogonal decomposition
$$
L^2(\mathbb{R}^3,\mathbb{C}^4)=L^-\oplus L^+,
$$
where $\mathscr{D}_c$ is positive defined on $L^+$ and negative defined on $L^-$. %Let $|\mathscr{D}_c|$ denote the absolute value of $\mathscr{D}_c$ and $|\mathscr{D}_c|^{\frac{1}{2}}$ denote its square root.
Let $E_c$ be the completion of $\mathcal{D}(|\mathscr{D}_c|^{1/2})$ under the following inner product
$$
(u_1,u_2)_c:= \left(|\mathscr{D}_c|^{1/2}u_1, |\mathscr{D}_c|^{1/2}u_2\right)_{L^2},
$$
The induced norm is denoted by $\| u\|_c:=(u,u)_c^{1/2}$. Moreover, we have $E_c\cong H^{1/2}(\mathbb{R}^3,\mathbb{C}^4)$ for any $c>0$. In fact, for any $u\in E_c$, we have 
$$
mc^2\|u\|_{L^2}^2\leq \| u\|_c^2,
$$
and
\[\min\{mc^2,c\}\|u\|_{H^{1/2}}^2  \leq  \| u\|_c^2\leq \max\{mc^2,c\}\|u\|_{H^{1/2}}^2 .\]
It is clear that the linear space $E_c$ possesses the following decomposition
$$
E_c=E_c^{-} \oplus E_c^{+}, \quad \text { where } E_c^{\pm}:=E_c \cap L^{\pm}.
$$
Denote $P_c^\pm=\frac{1}{2}\left(I \pm \frac{\mathscr{D}_c}{\left|\mathscr{D}_c\right|}\right)$ the orthogonal projections on $E_c$ with kernel $E_c^\mp$. In the standard model, a state $u\in E_c$ is superpositions of particles and its antiparticles, $u^+:=P_c^+u $ describes the state of Dirac fermions with positive energy, and $u^-:=P_c^-u $ describes its antiparticles with positive energy, which can cancel part of the energy of $u^+$. The projection can be extended to $L^q$ continuously as in the following Lemma, the proof can be find in \cite{dongdingguo}.
\begin{lemma}
Let $E^\pm_q:=E_c^\pm\cap L^q$ for $q\in (1,\infty)$. Then there holds
\[L^q=\text{cl}_q E^+_q\oplus \text{cl}_q E^-_q,\]
where $\text{cl}_q$ denotes the closure with respect to the norm in $L^q$. That is, there exists $\tau_q>0$ for every $q\in (1,\infty)$ such that
\[\tau_q\|u^\pm\|_{L^q}\leq \|u\|_{L^q}, \ \forall u\in E_c\cap L^q.\]
\end{lemma}
In the Fourier domain, for each $\xi \in \mathbb{R}^3$, the symbol matrix
$$
\hat{\mathscr{D}}_c(\xi):=   \mathcal{F} \mathscr{D}_c \mathcal{F}^{-1} =\left(\begin{array}{cc}
m c^2 I_2 & c \sigma \cdot \xi \\
c \sigma \cdot \xi & -m c^2 I_2
\end{array}\right)
$$
  is a Hermitian $4\times 4$-matrix with eigenvalues
$\pm \sqrt{m^2 c^4+c^2|\xi|^2}$.
It follows from a direct calculation that the unitary transformation $\mathbf{U}(\xi)$  diagonalizing $\hat{\mathscr{D}_c}(\xi)$
is given  by
$$
\mathbf{U}(\xi)=\frac{\left(m c^2+\lambda(\xi)\right) I_4+\beta c \alpha \cdot \xi}{\sqrt{2 \lambda\left(m c^2+\lambda(\xi)\right)}}=\Upsilon_{+} I_4+\Upsilon_{-} \beta \frac{\alpha \cdot \xi}{|\xi|}
$$
where $\lambda(\xi)=\sqrt{m^2 c^4+c^2|\xi|^2}, $ $\Upsilon_{\pm}=\sqrt{\frac{1}{2}\left(1 \pm m c^2 / \lambda(\xi)\right)}$.
Consequently, it is easy to see that
$$
\mathbf{U}^{-1}(\xi)=\frac{\left(m c^2+\lambda\right) I_4-\beta c \alpha \cdot \xi}{\sqrt{2 \lambda\left(m c^2+\lambda(\xi)\right)}}=\Upsilon_{+} I_4-\Upsilon_{-} \beta \frac{\alpha \cdot \xi}{|\xi|},
$$
$$
\mathbf{U}(\xi) \hat{\mathscr{D}_c}(\xi) \mathbf{U}^{-1}(\xi)=\lambda \beta.
$$
Therefore, we have 
\begin{align*}
\widehat{P_c^{\pm}u}(\xi)
=\frac{1}{2}\mathbf{U}^{-1}(\xi)(I_4\pm \beta)\mathbf{U}(\xi)\hat{u}(\xi)=\frac{1}{2}\left(I_4\pm\frac{mc^2}{\lambda}\beta\pm\frac{c}{\lambda}\alpha\cdot \xi\right)\hat{u}(\xi).
\end{align*}
Let us recall  the Foldy--Wouthuysen transformation is defined by $\mathbf{U}_{FW}:= \mathcal{F}^{-1}\mathbf{U}\mathcal{F}$,  which transforms
the free Dirac operator into the $2\times 2$-block form
$$
\mathbf{U}_{\mathrm{FW}} \mathscr{D}_c \mathbf{U}_{\mathrm{FW}}^{-1}=\left(\begin{array}{cc}
\sqrt{-c^2 \Delta+m^2 c^4}~I_2 & 0 \\
0 & -\sqrt{-c^2 \Delta+m^2 c^4}~I_2
\end{array}\right)=\beta\left|\mathscr{D}_c\right| .
$$

We shall recall the weal Young inequality \cite[Theorem 4.2]{ MR1817225} and the Gagliardo-Nirenberg inequality \cite[Theorem 1.1]{MR3428469}. 
These inequalities play a crucial role in subsequent proofs.
%参考文献
\begin{lemma}[\bf weak Young inequality]\label{lemm:2.1}
  Let $p, q, r \geq 1$ with $1 / p+1 / q+1 / r=2$.
  If $f \in L^p\left(\mathbb{R}^n\right),$ $ g \in L^q_w\left(\mathbb{R}^n\right)$ and $h \in L^r\left(\mathbb{R}^n\right)$, then we have
$$
\begin{aligned}
\left|\int_{\mathbb{R}^n} f(x)(g * h)(x) \mathrm{d} x\right|  \lesssim \|f\|_p\|g\|_{q,w}\|h\|_r.
\end{aligned}
$$
\end{lemma}

\begin{lemma}[\bf Gagliardo-Nirenberg inequality]\label{lemm:2.2}
  Let $1 \leqslant p, q, r \leqslant \infty$, then
$$
\|u\|_{L^p\left(\mathbb{R}^n\right)} \leq  C\left\|(-\Delta)^{\frac{m}{2}} u\right\|_{L^r\left(\mathbb{R}^n\right)}^\theta\|u\|_{L^q\left(\mathbb{R}^n\right)}^{1-\theta}
$$
holds for any $m \in \mathbb{R}^{+}$and $\theta \in[0,1]$ satisfying
$$
\frac{1}{p}=\left(\frac{1}{r}-\frac{m}{n}\right) \theta+(1-\theta) \frac{1}{q}.
$$
\end{lemma}

%这个应该是个直接的推论
A direct corollary of the Gagliardo–Nirenberg inequality is given by
\begin{corollary}\label{lemm:2.3}
  For $s\in[2,3]$ and $P\in L^{\infty}(\mathbb{R}^3)$, it holds that
  $$
  \int_{\mathbb{R}^3}P |{u}|^s\ d x\leq  C  \|(-\Delta)^{\frac{1}{4}}u\|_{L^2}^{3s-6}\|u\|_{L^2}^{6-2s},\quad \forall {u}\in H^{1/2}(\mathbb{R}^3,\mathbb{C}^4).
  $$
\end{corollary}

The following two lemmas are useful in our proof.
\begin{lemma}\label{lemm:2.4}

  Let $(K_1)$ and $(\Gamma_1)$ hold, then
  for any $ {u}  \in  H^{1/2}(\mathbb{R}^3,\mathbb{C}^4) $, it follows that\\
  $$\int_{\mathbb{R}^3 }  \Gamma *(K|{u}|^\kappa) K|{u}|^\kappa  d x \leq C  \| (-\Delta)^{\frac{1}{4}}u\|_{L^2}^2\|{u}\|_{L^2}^{2\kappa-2}.$$
\end{lemma}
%要加上Gamma的条件，这样的问题很多
\begin{proof}
  For   $\frac{1}{p}+\frac{1}{q}+\frac{1}{t}=2$, where $t=\frac{6}{14-6\kappa}$, by the weak Young inequality as in Lemma \ref{lemm:2.1}, we obtain
  $$\int_{\mathbb{R}^3 }  \Gamma *(K|{u}|^\kappa)K|{u}|^\kappa  \ d x \leq   C\int_{\mathbb{R}^3 }  \Gamma *  |{u}|^\kappa  |{u}|^\kappa  \ d x \leq  C \|{u}\|_{L^{p\kappa}}^\kappa\|{u}\|_{L^{q\kappa}}^\kappa.$$
Set
\begin{equation*}
  \begin{cases}
    \frac{1}{\kappa p}=(\frac{1}{2}-\frac{1}{6})\theta_1+\frac{1}{2}(1-\theta_1),\\
    \frac{1}{\kappa q}=(\frac{1}{2}-\frac{1}{6})\theta_2+\frac{1}{2}(1-\theta_2).
  \end{cases}
\end{equation*}
Then according to the Gagliardo-Nirenberg inequality as in Lemma \ref{lemm:2.2},
we arrive at
$$
\int_{\mathbb{R}^3 } \Gamma *(K|{u}|^\kappa)K|{u}|^\kappa  \ d x \leq  C  \| (-\Delta)^{\frac{1}{4}}u\|_{L^2}^2\|{u}\|_{L^2}^{2\kappa-2}.
$$
The proof is completed.
\end{proof}

\begin{lemma}\label{lemm:2.5}
  Let $(K_1)$ and $(\Gamma_1)$ hold, then   for any $ {u}, v,w\in H^{1/2}(\mathbb{R}^3,\mathbb{C}^4)$, it holds that
  $$
  \int_{\mathbb{R}^3 }\Gamma *(K|{u}|^\kappa)K|{u}|^{\kappa-2}|v||w|dx\leq  C   \|u\|_{L^2}^{2\kappa-3}\cdot \|u\|_{L^3}\cdot\|v\|_{L^2}\cdot\| w\|_{L^3}.
  $$
\end{lemma}
  \begin{proof}
    For $p$, $q$, $r>0$ satisfying    $\frac{1}{p}+\frac{1}{q}+\frac{1}{t}=2$, where $t=\frac{6}{14-6\kappa}, p=q=\frac{6}{3\kappa-1}$, by the weak Young inequality and the H\"older inequality, we obtain
    $$
    \begin{aligned}
      \int_{\mathbb{R}^3 } (\Gamma *|{u}|^\kappa)|{u}|^{\kappa-2}|v||w|dx\lesssim & \|u\|_{L^{p\kappa}}^\kappa\cdot\||u|^{\kappa-2}vw  \|_{L^q}\\
\lesssim&\|u\|_{L^{p\kappa}}^\kappa\cdot \|u\|_{L^2}^{k-2}\cdot\|v\|_{L^2}\cdot\| w\|_{L^3}.
    \end{aligned}
    $$
    By the interpolation inequality, we have
    $$
    \|u\|_{L^{p\kappa}}^\kappa\leq \|u\|_{L^2}^{\kappa-1}\cdot \|u\|_{L^3}.
    $$
    Therefore,
    $$
  \int_{\mathbb{R}^3 }  \Gamma *(K|{u}|^\kappa)K|{u}|^{\kappa-2}|v||w|dx\lesssim \|u\|_{L^2}^{2\kappa-3}\cdot \|u\|_{L^3}\cdot\|v\|_{L^2}\cdot\| w\|_{L^3}.
    $$
    This ends the proof. 
  \end{proof}

 The following lemma is useful to estimate the Lagrange multiplier .
The proofs of these two lemmas can be found in the appendix.
  \begin{lemma}\label{lemm:2.6}
  Let $P(x)\in L^\infty$, then  $\forall {u}  \in  H^{1/2}(\mathbb{R}^3,\mathbb{C}^4) $, with $\|{u}\|_{L^2}=1$,  ${u}=t w+u^- $, where $w=\frac{{u}_+}{\|{u}_+\|_{L^2}}, t=\sqrt{1-\|u^-\|_{L^2}^2}$, we have the following inequality
    $$
    \int_{\mathbb{R}^3 }P |{u}|^sdx\geq  C  \int_{\mathbb{R}^3 }P |w|^sdx-C   \|u^- \|_{L^2}^2\|w\|_{H^{1/2}}^2- C  \|u^- \|_{H^{1/2}}^2.
    $$
    For $w=\mathbf{U}_{\mathrm{FW}}^{-1}\left(\begin{array}{l}v \\ 0\end{array}\right)$ with $v \in H^1\left(\mathbb{R}^3, \mathbb{C}^2\right)$ and $\|v\|_{L^2}^2=1$, we have

    $$
    \int_{\mathbb{R}^3 } P |{u}|^s dx\geq  C  \int_{\mathbb{R}^3 } P |v|^sdx-C  \|\nabla v\|_{L^2}^s-C   \|u^- \|_{L^2}^2\|w\|_{H^{1/2}}^2- C  \|u^- \|_{H^{1/2}}^2.
    $$
  \end{lemma}

  \begin{lemma}\label{lemm:2.7}
    Let  $p\in [2,3]$, $\Gamma   $  satisfies $(\Gamma_1)$,
    set $$\mathscr{F}( u )=\Gamma * (K| u |^\kappa) K| u |^{\kappa-1},$$then
    $\mathscr{F}: H^{1/2}\to L^p$ is locally Lipschitz.  
  \end{lemma}

\section{The existence of normalized solutions}

In this section, we assume that $c$ large enough. Thus, we have $E_c\cong H^{1/2}$ and
 \[c\|u\|_{H^{1/2}}^2  \leq  \| u\|_c^2\leq  mc^2 \|u\|_{H^{1/2}}^2 .\]
 First, we introduce the energy functional $\mathcal{I}^c: E_c\rightarrow \mathbb{R}$  related to \ref{eq:Dirac} :
$$
\begin{aligned}
  \mathcal{I}^c({u})=&\int_{\mathbb{R}^3}({u}, \mathscr{D}_c {u})_{\mathbb{C}^4}dx-\frac{1  }{\kappa} \int_{\mathbb{R}^3 } \Gamma *(K|{u}|^\kappa)K|{u}|^\kappa d x -\frac{2}{s}\int_{\mathbb{R}^3}P |{u}|^s d x .
\end{aligned}
$$
By the spectral properties of $\mathscr{D}_c$, the functional $\mathcal{I}_c$ can be reformed as
$$
\begin{aligned}
  \mathcal{I}^c({u})
=&\|u^+\|_c^2-\| u^-\|_c^2-\frac{1  }{\kappa} \int_{\mathbb{R}^3 } \Gamma *(K|{u}|^\kappa)K|{u}|^\kappa  dx -\frac{2}{s}\int_{\mathbb{R}^3}P |{u}|^sdx.
\end{aligned}
$$
\begin{definition}
We say $(u, \omega)\in H^{1/2}(\mathbb{R}^3,\mathbb{C}^4)\times\mathbb{R}$ is a (weak) normalized solution of \ref{eq:Dirac}  if and only if for any $v\in \mathcal{C}_0^\infty(\mathbb{R}^3,\mathbb{C}^4)$, we have
$$
d  \mathcal{I}^c(u)[v]= 2\omega \Re \int_{\mathbb{R}^3}u\cdot vdx.
$$
Moreover, if the weak normalized solution $(u, \omega)\in H^{1}(\mathbb{R}^3,\mathbb{C}^4)\times\mathbb{R}$, then $(u,\omega)$ is called strong normalized solution of \ref{eq:Dirac}.
\end{definition}
It is clear that (weak) normalized solutions of \ref{eq:Dirac} correspond to critical points of the functional $\mathcal{I}^c$ restricted on the $L^2$-sphere
$$
S =\{u\in E_c: \|u\|_{L^2} =1 \}.
$$
Next, we will use the reduction method to study the normalized solutions of  \ref{eq:Dirac} .

\subsection{Maximization problem}
\quad For any $w\in E_c^+$ with $\|w\|_{L^2}=1 $, denote $W=\operatorname{span}\{w\}$ and set
$$
S_W=\{ u\in S: u^+\in W \}.
$$
Our  first step is to maximize the functional $\mathcal{I}^c$ on the space $S_W$. The tangent space of $S_W$ at $u\in S_W$ is given by$$
T_u(S_W)=\{h\in W\oplus E_c^-:\Re(u,h)_{L^2}=0 \}.
$$
We denote the projection of the gradient $d \mathcal{I}^c(u)$ on $T_u(S_W)$   by $d \mathcal{I}^c|_{S_W}(u)$, which is given by
$$
d \mathcal{I}^c|_{S_W}(u)[ h] = d \mathcal{I}^c(u)[h]- 2\omega(u)\Re(u,h)_{L^2},\quad \forall h\in W\oplus E_c^-,
$$
where $\omega(u)\in \mathbb{R}$ is such that $d \mathcal{I}^c|_{S_W}(u)\in T_u(S_W).$
We also denote 
$$S^+ =\{ u\in E_c^+: \|u\|_{L^2}=1 \} .$$
 For any $w \in S^+$, we say $\left\{{u}_n\right\} \subset   S_W$ is a Palais-Smale sequence of $\mathcal{I}^c$ on $S_W$ with value $e_c$ if 
  $$
  \mathcal{I}^c\left({u}_n\right) \rightarrow e_c \text { and }\left\| d \mathcal{I}^c  \big|_{S_W} \left({u}_n\right)\right\|   \rightarrow 0, \quad \text { as } \quad n \rightarrow+\infty .
  $$
  Our first results is about the compactness of Palais-Smale sequence of $\mathcal{I}^c$ on $  S_W$.
\begin{proposition}\label{prop:3.1}
If  $\left\{{u}_n\right\} \subset   S_W$ is a Palais-Smale sequence of $\mathcal{I}^c$ on $  S_W$ with value $e_c$,
then we have
  \begin{itemize}
  \item[(i)] $\{u_n\}$ is bounded in $E_c$.
  \item[(ii)]
  If $e_c>0$, then there exists $c_0>0$, such that for $c>c_0$, we have $$\liminf\limits_{n \rightarrow+\infty} \omega\left({u}_n\right)>0.$$ 
  \item[(iii)] Under the assumptions of (ii), the sequence $\left\{{u}_n\right\}$ is precompact in $   H^{1/2}(\mathbb{R}^3,\mathbb{C}^4)$.
  \end{itemize}
\end{proposition}
\begin{proof}
\begin{itemize}
\item[(i)] Since $u_n=tw+u_n^-$ and $\|u_n\|_{L^2}=1$, then we have $\|u_n^+\|_c^2\leq \|w\|_c^2$. Combine with 
  $$
  \mathcal{I}^c(u_n)\leq  \|u_n^+\|_c^2-\|u_n^-\|_c^2,
  $$
 we have $$\|u_n^-\|_c^2\leq \|w\|_c^2-\mathcal{I}^c(u_n),$$ 
 which implies the boundness of $\{u_n\}$.

\item[(ii)] It follows directly from $e_c>0$ that $\| u_n^+\| _{ c}^2 \geq \| u_n^-\| _{  c}^2$ for $n$ sufficiently large. Therefore, we have $$\|u_n^+\|_c^2\geq \frac{1}{2}\|u_n\|_c^2\geq \frac{c}{2}\|u_n\|_{H^{1/2}}^2.$$
By Lemma \ref{lemm:2.2} and $\{u_n\}\subset S_W$, we have 
\[\int_{\mathbb{R}^3} |u_n|^s dx\lesssim \left\|(-\Delta)^{1/4} u_n\right\|_{L^2}^{3s-6}\cdot \left\| u_n\right\|_{L^2}^{6-2s}\lesssim \|u_n\|_{H^{1/2}}^{3s-6},\]
\[\int_{\mathbb{R}^3} |u_n^+|^s dx\lesssim \left\|(-\Delta)^{1/4} u_n^+\right\|_{L^2}^{3s-6}\cdot \left\| u_n^+\right\|_{L^2}^{6-2s}\lesssim \|u_n\|_{H^{1/2}}^{3s-6}.\]
Combine with the Sobolev inequality
\[\int_{\mathbb{R}^3} |u_n|^sdx\lesssim \|u_n\|_{H^{1/2}}^s,\quad \int_{\mathbb{R}^3} |u_n^+|^sdx\lesssim \|u_n\|_{H^{1/2}}^s,\]
and set $t=(8-3s)/(6-2s)\in (0,1)$, we obtain
\[\int_{\mathbb{R}^3} |u_n|^s dx=\left(\int_{\mathbb{R}^3} |u_n|^s dx\right)^{st} \cdot \left(\int_{\mathbb{R}^3} |u_n|^s dx\right)^{s(1-t)} \lesssim \|u_n\|_{H^{1/2}}^{st}\cdot \|u_n\|_{H^{1/2}}^{(3s-6)(1-t)}=\|u_n\|_{H^{1/2}}^2, \]
\[\int_{\mathbb{R}^3} |u_n^+|^s dx=\left(\int_{\mathbb{R}^3} |u_n^+|^s dx\right)^{st} \cdot \left(\int_{\mathbb{R}^3} |u_n^+|^s dx\right)^{s(1-t)} \lesssim \|u_n\|_{H^{1/2}}^{st}\cdot \|u_n\|_{H^{1/2}}^{(3s-6)(1-t)}=\|u_n\|_{H^{1/2}}^2. \]
Thus, we have 
\[\int_{\mathbb{R}^3}P(x)|u_n|^{s-2}\Re (u_n,u_n^+)dx \lesssim \int_{\mathbb{R}^3}|u_n|^{s-1}|u_n^+|dx \lesssim \left(\int_{\mathbb{R}^3}|u_n|^s\right)^{\frac{s-1}{s}}\cdot \left(\int_{\mathbb{R}^3} |u_n^+|^s\right)^{\frac{1}{s}}\lesssim \|u_n\|_{H^{1/2}}^2.\]
Similarly,  
\[\int_{\mathbb{R}^3}\Gamma *(K|{u}_n|^\kappa)K|{u}_n|^{\kappa-2}\Re({u}_n, u_n^+)   dx\lesssim \|u_n\|_{H^{1/2}}^2.\]
Therefore, we get
$$
\begin{aligned}
\frac{1}{2} d \mathcal{I}^c\left({u}_n\right)\left[u_n^+\right]&=\omega\left({u}_n\right)\| u_n^+\| _{L^2}^2+o(1)\\
&=\| u_n^+\|^2- \int_{\mathbb{R}^3}\Gamma *(K|{u}_n|^\kappa)K|{u}_n|^{\kappa-2}\Re({u}_n, u_n^+)  dx- \int_{\mathbb{R}^3}P(x) |{u}_n|^{s-2}\Re({u}_n, u_n^+ ) dx\\
&\geq  c\|u_n\|_{H^{1/2}} ^2-C \|u_n\|_{H^{1/2}}^2>0,
\end{aligned}
$$
where $c$ is large enough.
Hence we get
$\liminf\limits_{n \rightarrow+\infty} \omega\left({u}_n\right)>0.$\\

\item[(iii)] By (i), we get $u_n\rightharpoonup u$ in $E_c$ (up to a subsequence). Since $\operatorname{dim} W=1$, hence $u_n^+ \rightarrow u^+ $ in $E_c$. In addition, we have
$$
\begin{aligned}
& o(1)=-\frac{1}{2} d \mathcal{I}^c\left({u}_n\right)\left[u_n^--u^- \right]+\omega\left({u}_n\right)\| u_n^--u^- \| _{L^2}^2 \\
\geq  &\| u_n^--u^- \|_c ^2 + \int_{\mathbb{R}^3}\Gamma *(K|{u}_n|^\kappa)K|{u}_n|^{\kappa-2}\Re({u}_n, u_n^- - u^-) dx
+\int_{\mathbb{R}^3}P |{u}_n|^{s-2}\Re({u}_n, u_n^- - u^-)dx\\
\geq  &\| u_n^--u^- \|_c^2+   \int_{\mathbb{R}^3}\Gamma *(K|{u}_n|^\kappa)K|{u}_n|^{\kappa-2}|u_n^- - u^-|^2 dx+\int_{\mathbb{R}^3}P |{u}_n|^{s-2}|u_n^- - u^-|^2 dx\\
-&  \int_{\mathbb{R}^3}\Gamma *(K|{u}_n|^\kappa)K|{u}_n|^{\kappa-2}(x) (|u^- |+| u_n^+|) |u_n^- - u^-| dx-\int_{\mathbb{R}^3}P |{u}_n|^{s-2}(|u^- |+| u_n^+|) |u_n^- - u^-|dx.
\end{aligned}
$$
For any $\varepsilon >0$, there exists $R>0$ large, such that
$$
\int_{\mathbb{R}^3\backslash B_R}|u^-|^3dx <\varepsilon.
$$
Combine with (i), we have
$$
\begin{aligned}
 & \int_{\mathbb{R}^3}\Gamma *(K|{u}_n|^\kappa)|{u}_n|^{\kappa-2}|u^- | |u_n^- - u^-| dx\\
  =&\int_{\mathbb{R}^3\backslash B_R}\Gamma *(K|{u}_n|^\kappa)|{u}_n|^{\kappa-2}|u^-| |u_n^- - u^-| dx\\
&+\int_{B_R}\Gamma *(K|{u}_n|^\kappa)|{u}_n|^{\kappa-2} |u^- | |u_n^- - u^-| dx\\
\leq &\| {u}_n\| _{L^3} \| {u}_n\| _{L^2}^{2\kappa-3}\| u^- \| _{L^3(\mathbb{R}^3\backslash B_R)}\| u_n^- -u^- \| _{L^2}\\
&+\| {u}_n\| _{L^3} \| {u}_n\| _{L^2}^{2\kappa-3}\| u^- \| _{L^3}\| (u_n^--u^- )\| _{L^2(B_R)}\\
\leq &\varepsilon + o(1).
\end{aligned}
$$
By the arbitrary of $\varepsilon$, one get
$$
\int_{\mathbb{R}^3}\Gamma *(K|{u}_n|^\kappa)K|{u}_n|^{\kappa-2} |u^- | |u_n^- - u^-| dx=o(1).
$$
Similarly, we also have
$$
\int_{\mathbb{R}^3}\Gamma *(K|{u}_n|^\kappa)K|{u}_n|^{\kappa-2} | u_n^+| |u_n^- - u^-| dx = o(1).
$$
By the H\"older inequality, we have
$$
\int_{\mathbb{R}^3}P |{u}_n|^{s-2} (|u^- |+| u_n^+|)  |u_n^- - u^-| dx    \lesssim \|{u}_n\|_{L^3}^{s-2}\left\| |u^- |+| u_n^+| \right\|_{L^{6/(5-s)}}\|u_n^- - u^-\|_{L^{6/(5-s)}}.
$$
This implies that
$$\int_{\mathbb{R}^3}P |{u}_n|^{s-2} (|u^- |+| u_n^+|)  |u_n^- - u^-|dx=o(1).$$
Hence
$$\| u_n^--u^- \| _{   H^{1/2}} =o(1).$$
This ends the proof of (iii).
\end{itemize}
\end{proof}

Combine Lemma \ref{lemm:2.3} and Lemma \ref{lemm:2.4}, and a similar argument in \cite[Proposition 3.2]{MR4327893}, one can get the following results which implies
the critical point of $ \mathcal{I}^c$ on $S_W$ at positive levels are strict local maxima.

\begin{lemma}\label{lemm:3.1}
If ${u}\in E_c$ is a critical point of $\mathcal{I}^c$ on $  S_W$ at a positive level, that is
$$d \mathcal{I}^c({u})[h]-2 \omega({u}  ) \Re({u}, h)_{L^2}=0,\quad \forall h \in W \oplus E_c^-,\quad \text{and}\quad \mathcal{I}^c({u}  )>0.$$
 Then
  $$
  d^2 \mathcal{I}^c({u})[h, h]-2 \omega({u})\|h\|_{L^2}^2<0.
  $$
  Hence  ${u}$ is a strict local maximum of $\mathcal{I}^c$ on $  S_W$.
\end{lemma}

For any $w \in S^+$, we consider the following maximization problem
\begin{align}\tag{3.1}\label{111}
  \lambda_W=\sup _{{u} \in   S_W} \mathcal{I}^c({u}) .
\end{align}

We have the following estimates on $ \lambda_W$.
\begin{lemma}\label{lemm:3.2}
  For any $w \in S^+$, we have
  $$
c-C \leq (c-C )\| w \|_{H^{1/2}}^2\leq  \lambda_W  \leq \|w\|_{ c}^2 .
  $$

\end{lemma}
\begin{proof}
Since $w\in S^+$, there exists $C>0$, such that $\|w\|_{H^{1/2}}^{3s-6}\leq C\|w\|_{H^{1/2}}^2$. Then by Corollary \ref{lemm:2.3} and Lemma \ref{lemm:2.4}, we have
\[ \int_{\mathbb{R}^3 }P |w|^sdx\leq  C  \|(-\Delta)^{1/4}w\|_{L^2}^{3s-6}\|w\|_{L^2}^{6-2s}\leq C\|w\|_{H^{1/2}}^{3s-6}\leq C\|w\|_{H^{1/2}}^{2},\]
and
\[\int_{\mathbb{R}^3 }\Gamma *(K|w|^\kappa)K|w|^\kappa  d x\leq C  \| (-\Delta)^{1/4}w\|_{L^2}^2\|{w}\|_{L^2}^{2\kappa-2}\leq C\|w\|_{H^{1/2}}^{2}.\]
Therefore, we get
  $$
  \begin{aligned}
\lambda_W \geq  \mathcal{I}^c(w)&=\|w\|^2- \frac{  1  }{\kappa} \int_{\mathbb{R}^3 }\Gamma * (K|w|^\kappa)K|w|^\kappa  d x-\frac{2 }{s} \int_{\mathbb{R}^3 }P |w|^sdx\\
  & \geq  \|w\|^2 - C \| w \|_{H^{1/2}}^2\\
  &\geq  (c-C )\| w \|_{H^{1/2}}^2\\
  &\geq     c-C .
  \end{aligned}
  $$
On the other hand, we have $\mathcal{I}^c({u}) \leq \| u^+ \| _{ c}^2 \leq \|w\|_{   c}^2$, for any ${u} \in   S_W$, that means $ \lambda_W  \leq \|w\|_{  c}^2$.
\end{proof}

We finish the maximization problem in view of the following proposition.
\begin{proposition}\label{prop:3.2}
For any $w \in S^+$,  there exists unique (up to a phase factor) $\varphi  (w) \in   S_W$, the strict global maximum of $\mathcal{I}^c$ on $  S_W$, namely
$$
\mathcal{I}^c\left(\varphi  (w)\right)=\sup _{{u} \in   S_W} \mathcal{I}^c({u})= \lambda_W.
$$
Moreover, we have
$$
d \mathcal{I}^c\left(\varphi  (w)\right)[h]-2 \omega\left(\varphi  (w)\right) \Re\left(\varphi  (w), h\right)_{L^2}=0, \quad \forall h \in W \oplus E_c^-,
$$
and
 the map $\varphi: v \in E_c^+ \backslash\{0\} \rightarrow \varphi (P(v))$, with $P(v)=v/\|v\|_{L^2}  \in S^+$, is smooth.
\end{proposition}
\begin{proof}
  By Lemma \ref{lemm:3.2}, $\lambda_W>0$,  Ekeland's variational principle yields there exists a Palais-Smale
  maximizing sequence $u_n$ of $\mathcal{I}^c$ on $S_W$ at a positive level. Hence $\|u_n-u\|_c^2\to 0$ in $E_c$, and $\omega(u_n)\to \omega(u)>0$, hence
  $$
  \mathcal{I}^c(u)=\sup _{{u} \in   S_W} \mathcal{I}^c({u}),
  $$
  and $\| d I  |_{S_W} \left({u}\right)\|=0.$
  In view of a similar argument in  \cite[Proposition 3.6]{MR4327893}, we get
  the uniqueness and  smoothness of $\varphi$ .
\end{proof}

\subsection{Minimization problem}
\quad We consider the smooth functionals $\mathcal{E}_c  :E_c^+ \backslash\{0\} \rightarrow \mathbb{R}$,  given by
$$
\mathcal{E}_c  (v)=\mathcal{I}^c\left(\varphi  (P(v))\right)=\sup _{{u} \in   S_W} \mathcal{I}^c({u}),
$$
where $W=\operatorname{span}\{w\}$, with $w=P(v) \in S^+$.
Consider the following  minimization problem
$$
e_c=\inf _{\substack{W \subset E_c^+ \\ \operatorname{dim} W=1}} \sup _{{u} \in   S_W} \mathcal{I}^c({u})=\inf _{w \in S^+} \mathcal{E}_c  (w),
$$
we can find that $e_c$ is bounded by a constant. This is the key idea to show the existence results of \ref{eq:Dirac}.
\begin{lemma}\label{lemm:3.3}
 For a give $c>0$ large enough, we have $ e_c\in \left(0,mc^2-C \right),$  where $C $ is a constant independent of $c$.
\end{lemma}

\begin{proof}
 We consider $w=\mathbf{U}_{\mathrm{FW}}^{-1}\left(\begin{array}{l}v \\ 0\end{array}\right) \in S^+$, with $v \in H^1\left(\mathbb{R}^3, \mathbb{C}^2\right)$ and $\|v\|_{L^2}^2=1$.  Since $\|w\|_{   H^{1/2}}^2=\|v\|_{   H^{1/2}}^2$ and
  $$0 \leq \|w\|_c^2-mc^2\|w\|_{L^2}^2
  =\int _{\mathbb{R}^3}\frac{c^2|\xi|^2}{\sqrt{c^2|\xi|^2+ m^2c^4 }+mc^2}|\hat{w}|^2(\xi) d\xi \leq \frac{1}{2m}\|\nabla v\|_{L^2}^2.$$

 By Lemma \ref{lemm:2.6},  we have for any ${u} \in S_W$
 $$
 \begin{aligned}
 \mathcal{I}^c({u})=&\| u^+ \|_c^2-\| u^- \|_c ^2-\frac{  1}{\kappa}\int_{\mathbb{R}^3 }\Gamma * (K|{u}|^\kappa)K|{u}|^\kappa  \ d x -\frac{2 }{s}\int_{\mathbb{R}^3 }P |{u}|^sdx\\
 \leq & \left(1-\|u^-\|_c^2\right)\|w\|_c^2-c\|u^-\|^2_{H^{1/2}}-  C  \int_{\mathbb{R}^3 } P |v|^sdx+C  \|\nabla v\|_{L^2}^s\\
 &+C   \|u^- \|_{L^2}^2\|w\|_{H^{1/2}}^2+ C  \|u^- \|_{H^{1/2}}^2\\
% \leq  &\|w\|^2-mc^2\|w\|_{L^2}^2-c\| u^- \| _{L^2}^2\|w\|_{   H^{1/2}}^2-c\| u^- \| _{   H^{1/2}}^2+mc^2 \\
% &- \frac{  1 }{s}\int_{\mathbb{R}^3 }P |v|^s dx+   C  \|\nabla v\|_{L^2}^s+   C   \|u^- \|_{L^2}^2\|w\|_{H^{1/2}}^2+   C  \|u^- \|_{H^{1/2}}^2\\
\leq & mc^2+\frac{1}{2m}\|\nabla v\|_{L^2}^2-(c-   C )\|u^- \|_{H^{1/2}}^2-(c-   C  )\| u^- \| _{L^2}^2\|w\|_{   H^{1/2}}^2    \\
  &  + C  \|\nabla v\|_{L^2}^s-  C \int_{\mathbb{R}^3 }P |v|^sdx.
 \end{aligned}
 $$
 Now, for any $\varepsilon>0$, we define $v_\varepsilon(x)=\varepsilon^{3 / 2} v(\varepsilon x)$, $w_\varepsilon=\mathbf{U}_{\mathrm{FW}}^{-1}\left(\begin{array}{c}v_\varepsilon \\ 0\end{array}\right) \in S^+$, and $W_\varepsilon=\operatorname{span} \{ w_\varepsilon\}$.
 Then by $(P_2)$, we have
 $$
e_c-mc^2\leq  \sup _{{u} \in S_{W_\varepsilon}} \mathcal{I}^c({u})-mc^2\lesssim \varepsilon^2\|\nabla v\|_{L^2}^2+    \varepsilon^s\|\nabla v\|_{L^2}^s- \varepsilon^{3s/2+\mu-3} \int_{\mathbb{R}^3 }P(x) |v|^sdx.
 $$
 Hence, by choosing proper $v$ and taking $\varepsilon>0$ sufficiently small, we may conclude that $e_c-mc^2<-C <0$.
 \end{proof}

\medskip
\noindent[\textbf{Proof of Theorem 1.1}]
  \textbf{Existence:}  By Ekeland's variational principle, there exists a   minimizing sequence $\left\{w_n\right\} \subset S^+$, namely $\mathcal{E}_c\left(w_n\right)=\mathcal{I}^c\left(\varphi\left(w_n\right)\right) \rightarrow e_c$ and $\| d \mathcal{E}_c\left(w_n\right)\|  \rightarrow 0$. Set ${u}_n:=\varphi\left(w_n\right)$, then we have
$$
\sup _{\|h\|_{   H^{1/2}}=1}\left|d \mathcal{I}^c\left({u}_n\right)[h]-2 \omega\left({u}_n\right) \Re\left({u}_n , h\right)_{L^2}\right| \rightarrow 0 .
$$
One can get the boundedness of $\omega({u}_n)$, we have that $\omega\left({u}_n\right) \rightarrow \omega_c<mc^2 $ (up to a subsequence). Since $$\left\|u_n^+\right\|_c^2-\left\|u_n^-\right\|_c^2\geq \mathcal{I}^c(u_n)>0,$$
we have
 $$\| u_n^-  \|_c ^2 \leq\|u_n^+\|_c^2\leq
 \| w_n\|_c ^2 \quad \text{and}\quad \| u_n^-  \|_{H^{1/2}} ^2 \leq mc
 \| w_n\|_{H^{1/2}} ^2.$$
Then by Lemma \ref{lemm:3.2},   we have 
$$
1=\| {u}_n\| _{L^2}^2 \leq \| {u}_n\| _{   H^{1/2}}^2 \leq  (c+1)\| w_n\| _{   H^{1/2}}^2 \leq \frac{c+1}{c-C    }(e_c+o(1)) .
$$ 
Therefore, $\left\{{u}_n\right\}$ is a Palais-Smale sequence for the functional
$$
\mathcal{I}^c_\omega({u})=\mathcal{I}^c({u})-\omega_c\|{u}\|_{L^2}^2,
$$
satisfying
$$
0<\inf _n\| {u}_n\| _{   H^{1/2}}^2 \leq \sup _n\| {u}_n\| _{   H^{1/2}}^2<+\infty .
$$
We may assume $u_n\rightharpoonup u$ in $E_c$, and $u_n\to u$ in $L^p_{loc}$, where $p\in (1,3)$. 
Since
\[\mathcal{I}^c(u_n)-\frac{1}{2}d\mathcal{I}^c (u_n)[u_n]\geq 0,\]
we have $\omega_c\leq e_c<mc^2-C$.
 Then
$$
\begin{aligned}
  \left\|u_n^{+}-u^{+}\right\|_c^2=&\left(d\mathcal{I}^c \left(u_n\right)-d\mathcal{I}^c(u), u_n^{+}-u^{+}\right)-\int_{\mathbb{R}^3}\left(P \left|u_n\right|^{s-2} u_n-P |u|^{s-2} u, u_n^{+}-u^{+}\right) dx\\
  -&\int_{\mathbb{R}^3}\left(\Gamma *(K|u_n|^\kappa)K|u_n|^{\kappa-2}u_n-\Gamma *(K|u|)K|u|^{\kappa-2}u, u_n^{+}-u^{+}\right) d x+\omega_c\|u_n^{+}-u^{+}\|_{L^2}^2\\
  \leq &\omega_c\|u_n^{+}-u^{+}\|_{L^2}^2+o(1).
\end{aligned}
$$
By Lemma \ref{lemm:3.3}, we have
$\left\|u_n^{+}-u^{+}\right\|_c^2=o(1)$ . Similarly, we have  $\left\|u_n^{-}-u^{-}\right\|_c^2=o(1)$, which means
$$\left\|u_n-u\right\|_c\rightarrow 0, \ \text{as}\ n\rightarrow \infty.$$
This implies $\|u\|_{L^2} =1$. Thus, we have $(u,\omega_c)\in E\times (0,mc^2)$ is a (weak) normalized solution of \ref{eq:Dirac}.

\medskip
\textbf{Regularity:}
We use the bootstrap argument to obtain the regularity of solutions of \ref{eq:Dirac}. Set $$ u _1=(\mathscr{D}_c-\omega_c)^{-1}\left(\Gamma *(K|u |^\kappa)K|u |^{\kappa-2}u \right), \quad  u _2=(\mathscr{D}_c-\omega_c)^{-1}\left(P| u |^{s-2} u\right), $$
then by Lemma \ref{lemm:2.7}, we have $ u _1\in W^{1,p}\left(\mathbb{R}^3,\mathbb{C}^4\right)$ for $p\in [2,3]$, and $u _2\in W^{1,\frac{3}{s-1}}\left(\mathbb{R}^3,\mathbb{C}^4\right)$. Hence by the Sobolev embedding theorem, we have $$ u _1\in \bigcap_{p\geq  2}L^p \left(\mathbb{R}^3,\mathbb{C}^4\right)\quad \text{and}\quad u _2\in L^{\frac{3}{s-2}} \left(\mathbb{R}^3,\mathbb{C}^4\right).$$ Then we have  $ u \in  L^{\frac{3}{s-2}}\left(\mathbb{R}^3,\mathbb{C}^4\right)$,
by $(K_1)$, $(\Gamma_1)$ and the Young inequality, we have $$ u _1\in W^{1,p}\left(\mathbb{R}^3,\mathbb{C}^4\right),\quad  u _2\in W^{1,q}\left(\mathbb{R}^3,\mathbb{C}^4\right),$$ where $p=\frac{3}{(4-3\kappa)+(2\kappa-1)(s-2)}>3$, $q=\frac{3}{(s-1)(s-2)}$.
Set $q_0=\frac{3}{s-2}$, and
$$
 q_n=\frac{3q_{n-1}}{(s-1)(3-q_n)},
$$
then $ u _2\in W^{1, q_n}\left(\mathbb{R}^3,\mathbb{C}^4\right)$, one can prove there exist $q_n > 3$, thus we  get $ u _2\in W^{1, 3}\left(\mathbb{R}^3,\mathbb{C}^4\right)$. Then by the Sobolev embedding theorem, we have  $ u _2\in L^p\left(\mathbb{R}^3,\mathbb{C}^4\right)$, for any $p\geq 2$.
Consequently, we get $ u \in L^p\left(\mathbb{R}^3,\mathbb{C}^4\right),$ which implies $ u \in W^{1,p}\left(\mathbb{R}^3,\mathbb{C}^4\right)$, for any $p\geq 2$.

\medskip
We now turn to prove the boundedness of $\omega_c - mc^2$ as 
$c\to \infty. $
 Let $\{u_c=( f_c, g_c)\}$ be a critical point of the functional  $\mathcal{I}^c$ on $S$ with Lagrange multiplier $\omega_c$,
 where
$f_c, g_c\in H^{1/2}(\mathbb{R}^3,\mathbb{C}^2)$.
\begin{lemma}\label{lemm:4.1}
  $\{u_c\}$ is uniformly bounded in  $H^1\left(\mathbb{R}^3,\mathbb{C}^4\right)$ with respect to $c$.
\end{lemma}
\begin{proof}
 We have  proved $u_c\in H^1\left(\mathbb{R}^3,\mathbb{C}^4\right)$. According to Lemma \ref{lemm:2.1} and  Lemma \ref{lemm:2.2}, it follows that
  $$
\begin{aligned}
  \|\mathscr{D}_cu_c\|_{L^2}^2=& m^2c^4\|u_c\|_{L^2}^2+c^2\|\nabla u_c\|_{L^2}^2=  \left\|  (\Gamma * K| u_c |^\kappa) K| u_c |^{\kappa-2} u_c +   P | u_c |^{s-2} u_c +\omega_c  u_c   \right\|_{L^2}^2 \\
\leq&  \| (\Gamma * K| u_c |^\kappa) K| u_c |^{\kappa-2} u_c\|_{L^2}^2+\| P | u_c |^{s-2} u_c \|_{L^2}^2+\| \omega_c  u_c \|_{L^2}^2\\
&+ \langle (\Gamma * K| u_c |^\kappa) K| u_c |^{\kappa-2} u_c,   P | u_c |^{s-2} u_c \rangle_{L^2}\\&+\langle (\Gamma * K| u_c |^\kappa) K| u_c |^{\kappa-2} u_c, \omega_c  u_c \rangle_{L^2}+\langle  P | u_c |^{s-2} u_c ,\omega_c  u_c \rangle_{L^2}\\
\leq& m^2c^4\|u_c\|_{L^2}^2+\mathcal{C}(\|\nabla u_c\|_{L^2}^{3s-6}+\|\nabla u_c\|_{L^2}^{2}+ mc^2\|\nabla u_c\|_{L^2}^{\frac{3s}{2}-3}+mc^2\|\nabla u_c\|_{L^2}+\|\nabla u_c\|_{L^2}^{\frac{s}{2}})
\end{aligned}
 $$
 Hence $\|\nabla u_c\|_{L^2}^2$ is bounded uniformly, which implies $\{u_c\}$ is bounded in  $H^1$ uniformly.
\end{proof}
 It is worth to mention that the expansion of the inner product plays a crucial role in this context because the term $\omega_c u_c$ can be incorporated into the expression $mc^2u_c$. From Lemma \ref{lemm:3.3}, we get $$\limsup\limits_{c\to \infty} (\omega_c-mc^2)< 0.$$On the other hand, $d \mathcal{I}^c  |_{S}(u_c)[u_c^+]=0$ yields
  $$
  \|u_c^+\|_c^2-\int_{\mathbb{R}^3 } \Gamma *(K|{u_c}|^\kappa)K|{u_c}|^{\kappa-2}\Re(u_c,u_c^+)  d x -\int_{\mathbb{R}^3}P |{u_c}|^{s-2}\Re (u_c,u_c^+)d x= \omega_c\|u_c^+\|_{L^2}^2.
  $$
  By using Lemma \ref{lemm:4.1}, we obtain that there exists a constant $\mathcal{C}$ independent of $c$, such that
  $$
  \int_{\mathbb{R}^3 } (\Gamma *K|{u_c}|^\kappa)(x)K|{u_c}|^{\kappa-2}(x)\Re(u_c,u_c^+) d x + \int_{\mathbb{R}^3}P |{u_c}|^{s-2}\Re (u_c,u_c^+) dx\leq \mathcal{C}.
  $$
  Hence, we get
  $$
  (\omega_c-mc^2)\|u_c^+\|_{L^2}^2 \geq \omega_c\|u_c^+\|_{L^2}^2- \|u_c^+\|_c^2\geq  -\mathcal{C},
  $$
  which implies $-\infty < \liminf\limits_{c\to \infty} (\omega_c-mc^2).$ Therefore, we have
\begin{equation}\tag{3.2}\label{4.0}
-\infty < \liminf\limits_{c\to \infty} (\omega_c-mc^2)\leqslant \limsup\limits_{c\to \infty} (\omega_c-mc^2)< 0.
\end{equation}
This ends the proof of Theorem \ref{them:1.1}.

\section{Nonrelativistic Limits  of the   Normalized Solutions}

This section is devoted to the limit of the solutions $\{u_c=(f_c,g_c)^T\}$ of the equation \ref{eq:Dirac} as $c\rightarrow \infty$. We will prove that the first components of the family $\{u_c\}$ converge to a solution of nonlinear Sch\"odinger equation \eqref{NSE} and the second components converge to zero. Actually, normalized solutions for nonlinear Dirac equations exhibit numerous advantageous properties, allowing for straightforward verification of various nonrelativistic limit results. Specifically, we can directly establish that the family $\{u_c\}$is bounded away from zero in $H^1(\mathbb{R}^3,\mathbb{C}^4)$ by utilizing the Sobolev Embedding Theorem.  %Since $u_c$ is $L^2$-normalized, then we have
%the functional  $\mathcal{I}^c|_S$.
\begin{lemma}\label{lemm:4.2}
  The family $\{g_c\}$ converges to $0$ in $H^1(\mathbb{R}^3,\mathbb{C}^2)$, as $c\to \infty.$ Moreover, there holds
  \[\|g_c\|_{H^1}=\mathcal{O}\left(\frac{1}{c}\right), \quad \text{as}\quad c\rightarrow \infty.\]
\end{lemma}
\begin{proof}
 We can rewrite the equation as follows:
  \begin{equation}\tag{4.1}\label{4}
    \begin{aligned}
      -i c \sigma \cdot \nabla f_c-m c^2 g_c-\omega_c g_c=P |u_c|^{s-2} g_c+\Gamma * (K |{u_c}|^\kappa) K|{u_c}|^{\kappa-2}{g_c}, \\
    \end{aligned}
  \end{equation}
  \begin{equation}\tag{4.2}\label{5}
    \begin{aligned}
      -i c\sigma \cdot \nabla g_c+m c^2 f_c-\omega_c f_c=P |u_c|^{s-2} f_c+\Gamma *( K |{u_c}|^\kappa) K|{u_c}|^{\kappa-2}{f_c}.
    \end{aligned}
  \end{equation}
Dividing \eqref{5} by $c$, then its $L^2$-norm squared reads
$$
\begin{aligned}
\|\nabla g_c\|_{L^2}^2= \left\|  -\frac{mc^2-\omega_c}{c} f_c+\frac{1}{ c}P  |u_c|^{s-2} f_c+\frac{1}{c}\Gamma *  ( K |{u_c}|^\kappa) K|{u_c}|^{\kappa-2}{f_c}  \right\|_{L^2}^2.
\end{aligned}
$$
Therefore,  by Lemma \ref{lemm:4.1} and \eqref{4.0}, we have

$$
\begin{aligned}
\|\nabla g_c\|_{L^2}^2\lesssim \frac{(mc^2-\omega_c)^2}{c^2}\|f_c\|_{L^2}^2+ \frac{1}{c^2}\|\nabla u_c\|_{L^2}^{2s-2}+\frac{1}{c^2}\|\nabla u_c\|_{L^2}^2\lesssim \frac{1}{c^2}.
\end{aligned}
$$
Dividing \eqref{4} by $c^2$ and using Lemma \ref{lemm:4.1}, we get
$$
\begin{aligned}
\frac{(mc^2+\omega_c)}{c^2}\|g_c\|_{L^2}\lesssim \frac{1}{c}\|\nabla f_c\|_{L^2}+\frac{1}{c^2}\|u_c\|_{L^{2s-2}}^{s-1}+\frac{1}{c^2}\|\nabla u_c\|_{L^2}\lesssim\frac{1}{c},
\end{aligned}
$$
which implies that
$$
\|g_c\|_{H^1}= \mathcal{O}\left(\frac{1}{c}\right), \,\,\, as \,\,\, c\to \infty.
$$
This ends the proof.
\end{proof}

\medskip
It is clear that a function $h\in H^1(\mathbb{R}^3,\mathbb{C}^2)$ is a solution of the nonlinear Schr\"odinger equation \eqref{NSE} if and only if $\|h\|_{L^2}=1$ and it is a critical point of the following functional $\mathcal{J}:H^1(\mathbb{R}^3,\mathbb{C}^2)\rightarrow \mathbb{R}$ which is defined by
$$
\mathcal{J}(h):=\frac{1}{2}\int_{\mathbb{R}^3}|\nabla h|^2 dx+ \frac{\nu}{2}\int_{\mathbb{R}^3}|h|^2dx-\frac{2m}{s}\int_{\mathbb{R}^3}P |h|^sdx-\frac{m}{\kappa}\int_{\mathbb{R}^3} \Gamma * (K |{h}|^\kappa) K|h|^{\kappa}dx.
$$
In what follows, we consider subsequences $\{u_{c_n}=\left(f_{c_n}, g_{c_n}\right)^T\}$ of the family of solutions. Denote $$e_{n}=\inf _{\substack{W \subset E^+ \\ \operatorname{dim} W=1}} \sup _{{u} \in   S_W} \mathcal{I}^{c_n}({u}).$$ Then we will prove that any subsequence $\{f_{c_n}\}$ of the family is a $(PS)$-sequence of the functional $\mathcal{J}$ in $H^1(\mathbb{R}^3,\mathbb{C}^2)$, where $c_n\rightarrow \infty$ as $n\rightarrow \infty$.

\begin{lemma}\label{lemm:4.3}
 The sequence $\left\{f_{c_n}\right\}$ is a $(PS)$-sequence of the functional $\mathcal{J}\in \mathcal{C}^1(H^1(\mathbb{R}^3,\mathbb{C}^2),\mathbb{R})$.

\end{lemma}
\begin{proof}
 Since $u_{c_n}=(f_{c_n},g_{c_n})^T$ is a solution of the nonlinear Dirac equation \ref{eq:Dirac} with $c=c_n$.
Then the equation \eqref{4} implies
  $$
  -i \sigma \cdot \nabla\left( -i c_n\sigma \cdot \nabla f_{c_n}-m c_n^2 g_{c_n}-\omega_{c_n} g_{c_n}\right)=   -i \sigma \cdot \nabla\ \left( P |u_{c_n}|^{s-2} g_{c_n}+\Gamma *( K |{u_{c_n}}|^\kappa) K|{u_{c_n}}|^{\kappa-2}{g_{c_n}}\right),
  $$
that is
\begin{equation}\tag{4.3}\label{ee1}
  \begin{aligned}
    -\Delta f_{c_n}=&\frac{mc_n^2+\omega_{c_n}}{c_n^2}\left( -(m c_n^2- \omega_{c_n}) f_{c_n}+ P |u_{c_n}|^{s-2} f_{c_n}+\Gamma * (K|{u_{c_n}}|^\kappa) K |{u_{c_n}}|^{\kappa-2}{f_{c_n}} \right)\\
    &-\frac{i}{c_n} \sigma \cdot \nabla\left( P |u_{c_n}|^{s-2} g_{c_n}+\Gamma *( K |{u_{c_n}}|^\kappa) K|{u_{c_n}}|^{\kappa-2}{g_{c_n}}\right)  .
  \end{aligned}
\end{equation}
Set
$$
a_n:=\frac{m^2c_n^4-\omega_{c_n}^2}{c_n^2}\to \nu,\quad b_n=\frac{mc_n^2+\omega_{c_n}}{c_n^2}\to 2m,\quad \text{as}\quad n\rightarrow \infty.
$$
 Multiplying \eqref{ee1} by $\phi$ and integrating over $\mathbb{R}^3$, we get
\begin{equation}\tag{4.4}\label{7}
  \begin{aligned}
    \Re\int_{\mathbb{R}^3}-\Delta f_{c_n}\cdot \phi dx+a_n\Re\int_{\mathbb{R}^3} f_{c_n}\cdot\phi dx =  L_1+L_2.
  \end{aligned}
\end{equation}
where
\[L_1:= b_n\Re\int_{\mathbb{R}^3}\left(P |u_{c_n}|^{s-2} f_{c_n}+\Gamma * (K|{u_{c_n}}|^\kappa) K |{u_{c_n}}|^{\kappa-2}{f_{c_n}} \right)\cdot \phi dx,\]
\[L_2:= - \Re\int_{\mathbb{R}^3}\frac{i}{c_n} \sigma \cdot \nabla\left( P |u_{c_n}|^{s-2} g_{c_n}+\Gamma *( K |{u_{c_n}}|^\kappa) K |{u_{c_n}}|^{\kappa-2}{g_{c_n}}\right)\cdot  \phi dx.\]
Then we deduce from Lemma \ref{lemm:4.1} that
$$
\begin{aligned}
  L_2\lesssim \frac{1}{c_n}\|\nabla\phi\|_{L^2}\left\| P |u_{c_n}|^{s-2} g_{c_n}+\Gamma *( K |{u_{c_n}}|^\kappa) K |{u_{c_n}}|^{\kappa-2}{g_{c_n}}   \right\|_{L^2} \lesssim \frac{1}{c_n}\to 0 ,\quad as \,\,\,n\to \infty.
\end{aligned}
$$
By Lemma \ref{lemm:4.2}, we obtain
$$
\begin{aligned}
   \left|\int_{\mathbb{R}^3}\left( P |u_{c_n}|^{s-2}  f_{c_n} -  P | f_{c_n}|^{s-2} f_{c_n}  \right)\cdot \phi dx\right|\lesssim \int_{\mathbb{R}^3}|g_{c_n}|^{s-2}| f_{c_n}||\phi|dx \to 0 ,\quad as \,\,\,n\to \infty.
\end{aligned}
$$
By Lemma \ref{lemm:2.1} and the H\"older inequality, we get
$$
\begin{aligned}
   &\left|\int_{\mathbb{R}^3}\left(\Gamma * (K|{u_{c_n}}|^\kappa) K |{u_{c_n}}|^{\kappa-2}{f_{c_n}}-\Gamma * (K|{f_{c_n}}|^\kappa) K |{f_{c_n}}|^{\kappa-2}{f_{c_n}} \right)\cdot \phi dx\right|\\
   &\lesssim \left|\int_{\mathbb{R}^3}\left( \Gamma * (K \left||{u_{c_n}}|^\kappa-| f_{c_n}|^\kappa\right|) K |{u_{c_n}}|^{\kappa-1}\right)|\phi|dx\right|\\
   &\quad  + \left|\int_{\mathbb{R}^3}\left( \Gamma * (K|f_{c_n}|^\kappa) K \left| |u_{c_n}|^{\kappa-2}-|f_{c_n}|^{\kappa-2} \right|   \right)|f_{c_n} | | \phi|dx\right|\\
   &\lesssim \left|\int_{\mathbb{R}^3}\left( \Gamma * (K|u_{c_n}|^{\kappa-1}|g_{c_n}|   ) K|{u_{c_n}}|^{\kappa-1}\right)|\phi|dx\right|\\
   &\quad + \left|\int_{\mathbb{R}^3}\left( \Gamma * (K | f_{c_n}|^\kappa) K \left| g_{c_n}\right|^{\kappa-2}   \right)| f_{c_n}||\phi|dx\right|\\
   &\lesssim \|g_{c_n}\|_{L^2}\|\phi\|_{L^2} \left\| u_{c_n}\right\|_{L^{\frac{6\kappa-6}{3\kappa-4}}}^{2\kappa-2}+\left\|f_{c_n}\right\|_{L^{\frac{6\kappa}{3\kappa-1}}}^\kappa\|g_{c_n}\|_{L^2}^{\kappa-2}\|f_{c_n}\|_{L^2}\|\phi\|_{L^3}\\
   &\quad \to 0 \quad as\,\,\, n\to \infty.
\end{aligned}
$$
This implies
\begin{equation}\tag{4.5}\label{8}
  L_1=b_n\Re\int_{\mathbb{R}^3} P | f_{c_n}|^{s-2} f_{c_n}  \cdot \phi dx + b_n\Re\int_{\mathbb{R}^3} \Gamma * (K|{f_{c_n}}|^\kappa) K |{f_{c_n}}|^{\kappa-2}{f_{c_n}}\cdot \phi dx+o(1).
\end{equation}
Consider the functional $\mathcal{J}_n\in \mathcal{C}^2\left(H^1(\mathbb{R}^3,\mathbb{C}^2),\mathbb{R}\right)$  defined by:
$$
\mathcal{J}_n(h):=\frac{1}{2}\int_{\mathbb{R}^3}|\nabla h|^2 dx+\frac{a_n}{2}\int_{\mathbb{R}^3}| h |^2dx-\frac{b_n}{ s}\int_{\mathbb{R}^3}P |h|^sdx- \frac{b_n}{ 2\kappa}\int_{\mathbb{R}^3}(\Gamma * (K |{h}|^\kappa) K |{h}|^{\kappa} dx.
$$
Combining (\ref{7}) and (\ref{8}), we obtain
$$
\sup _{\|\phi\|_{H^1}\leq  1}\left\langle\mathcal{J}_n'\left( f_{c_n}\right), \phi\right\rangle \rightarrow 0 \text { as } n \rightarrow \infty.
$$
It is clear that $$\sup _{\|\phi\|_{H^1}\leq 1}\left\langle\mathcal{J}'( f_{c_n})-\mathcal{J}_n'\left( f_{c_n}\right), \phi\right\rangle\to 0\quad \text{as} \quad n\rightarrow \infty,$$ then we have
$$
\sup _{\|\phi\|_{H^1}\leq 1}\left\langle\mathcal{J}'\left( f_{c_n}\right), \phi\right\rangle\rightarrow 0\,\,\, as\,\,\, n\to \infty.
$$
By Lemma \ref{lemm:4.1} and assumptions of $p$ and $\kappa$, we have $\left\{\mathcal{J}(f_{c_n})\right\}$ is bounded. Therefore, $\left\{f_{c_n}\right\}$ is a $(PS)$-sequence of the functional $\mathcal{J}$.
\end{proof}

\medskip
To study the ground states of nonlinear Schr\"odinger equations \eqref{NSE} under the $L^2$-constraint, we introduce the corresponding energy functional $\mathcal{E}\in \mathcal{C}^1(H^1(\mathbb{R}^3,\mathbb{C}^2),\mathbb{R})$ defined by 
$$ \mathcal{E}(g)=\frac{1}{2}\int_{\mathbb{R}^3}|\nabla g|^2 dx- \frac{m }{\kappa} \int_{\mathbb{R}^3 } \Gamma *(K|{g}|^\kappa)K|{g}|^\kappa d x - \frac{2m}{s}\int_{\mathbb{R}^3}P |{g}|^s d x.
    $$
    It is clear that critical points of $\mathcal{E}$ on the $L^2$-sphere lead to the solutions of the equations \eqref{NSE}. 
    
    \begin{definition}
    We say $h\in H^1(\mathbb{R}^3,\mathbb{C}^2)$ is a ground state of \eqref{NSE} if it is a critical point of $ \mathcal{E}$ and $$ \mathcal{E}(h)=\inf_{\|g\|_{L^2}=1}  \mathcal{E}(g).$$
    \end{definition}
The next two lemmas show that the limit function $h$ is a ground state of \eqref{NSE}.
% by estimating the energy of $u_c$ and $h$. 
\begin{lemma}\label{lem443}
 The family of solutions $\left\{u_{c}=(f_c,g_c)^T\right\}$ satisfies  $$\displaystyle\mathcal{I}^c({u_c}) 
    =\frac{1}{m} \mathcal{E}(f_c) +mc^2 + o(1)\quad \text{as} \quad c\rightarrow \infty.$$
\end{lemma}
\begin{proof}
Lemma \ref{lemm:4.2} implies $\|\nabla g_c\|_{L^2}=\mathcal{O}\left(\frac{1}{c}\right)$. Dividing \eqref{4} by $c^2$, we get
$$
\left\|2mg_c + \frac{i\sigma\cdot\nabla f_c}{c} \right\|_{L^2} = \mathcal{O}\left(\frac{1}{c^3}\right).
$$
Then
\begin{equation*}
  \begin{aligned}
    \mathcal{I}^c({u_c})&=\int_{\mathbb{R}^3}({u_c}, \mathscr{D}_c {u_c})_{\mathbb{C}^4}dx-\frac{1  }{\kappa} \int_{\mathbb{R}^3 } \Gamma *(K|{u_c}|^\kappa)K|{u_c}|^\kappa d x -\frac{2}{s}\int_{\mathbb{R}^3}P |{u_c}|^s d x \\
 &=\int_{\mathbb{R}^3}\left(mc^2|f_c|^2 - mc^2|g_c|^2 - 2\Re(ic\sigma\cdot \nabla f_c, g_c)           \right)\\
 &-\frac{1  }{\kappa} \int_{\mathbb{R}^3 } \Gamma *(K|{u_c}|^\kappa)K|{u_c}|^\kappa d x -\frac{2}{s}\int_{\mathbb{R}^3}P |{u_c}|^s d x.
  \end{aligned}
\end{equation*}
Since  
$$
g_c = -\frac{i c \sigma \cdot \nabla f_c + P |u_c|^{s-2} g_c+\Gamma * (K |{u_c}|^\kappa) K|{u_c}|^{\kappa-2}{g_c}}{mc^2 + \omega_c},
$$
one has
$$
\int_{\mathbb{R}^3}\Re(ic\sigma\cdot \nabla f_c, g_c)= -\frac{1}{2m} \int_{\mathbb{R}^3}|\nabla f_c|^2  + o(1).
$$
%$$
%\int_{\mathbb{R}^3}(|f_c|^2 + |g_c|^2) = 1,
%$$
then
\begin{equation*}
  \begin{aligned}
    \mathcal{I}^c({u_c})
 =&\int_{\mathbb{R}^3}\left(mc^2|f_c|^2 - mc^2|g_c|^2  -  2\Re(ic\sigma\cdot \nabla f_c, g_c)           \right)\\
 &-\frac{1  }{\kappa} \int_{\mathbb{R}^3 } \Gamma *(K|{u_c}|^\kappa)K|{u_c}|^\kappa d x -\frac{2}{s}\int_{\mathbb{R}^3}P |{u_c}|^s d x\\
  =& mc^2 - \int_{\mathbb{R}^3} (2mc^2|g_c|^2 +  \Re(ic\sigma\cdot \nabla f_c, g_c)) + \frac{1}{2m}\int_{\mathbb{R}^3}|\nabla f_c|^2  \\
 &-\frac{1  }{\kappa} \int_{\mathbb{R}^3 } \Gamma *(K|{u_c}|^\kappa)K|{u_c}|^\kappa d x -\frac{2}{s}\int_{\mathbb{R}^3}P|{u_c}|^s d x + o(1).\\ 
  \end{aligned}
\end{equation*}
Notice
$$
\int_{\mathbb{R}^3} |2mc^2|g_c|^2 + \Re (ic\sigma\cdot \nabla f_c, g_c)|\leq  c^2\|g_c\|_{L^2} \left\|2mg_c + \frac{i\sigma\cdot\nabla f_c}{c} \right\|_{L^2}\leq  \mathcal{O}\left(\frac{1}{c}\right),
$$
and
$$
\frac{1  }{\kappa} \int_{\mathbb{R}^3 } \Gamma *(K|{u_c}|^\kappa)K|{u_c}|^\kappa d x + \frac{2}{s}\int_{\mathbb{R}^3}P |{u_c}|^s d x = \frac{1  }{\kappa} \int_{\mathbb{R}^3 } \Gamma *(K|{f_c}|^\kappa)K|{f_c}|^\kappa d x + \frac{2}{s}\int_{\mathbb{R}^3}P |{f_c}|^s d x + o(1),
$$
then
\begin{equation*}
  \begin{aligned}
    \mathcal{I}^c({u_c}) &= \frac{1}{2m}\int_{\mathbb{R}^3}|\nabla f_c|^2 - \frac{1  }{\kappa} \int_{\mathbb{R}^3 } \Gamma *(K|{f_c}|^\kappa)(x)K(x)|{f_c}|^\kappa d x - \frac{2}{s}\int_{\mathbb{R}^3}P(x) |{f_c}|^s d x + mc^2 + o(1)\\
    & =\frac{1}{m} \mathcal{E}(f_c) +mc^2 + o(1).
  \end{aligned}
\end{equation*}  This ends the proof.
\end{proof}
%% We  now estimate the uniform boundedness %and the exponential decay properties
%of $u_c.$

\medskip

\noindent [\textbf{Proof of Theorem 1.2}]
From Lemma \ref{lemm:4.1} and \ref{lemm:4.3}, we have $\{f_{c_n}\}$ is a bounded $(PS)$-sequence of $\mathcal{J}$. Then up to a subsequence, we have for some $e_0>0$, 
\[\mathcal{J}(f_{c_n})\rightarrow e_0 \quad \text{as}\quad n\rightarrow \infty,\]
and there is a function $h\in H^1\left(\mathbb{R}^3,\mathbb{C}^2\right)$, such that $ f_{c_n}\rightharpoonup h$ in $H^1$. Denote the bilinear functional $\mathcal{B}: H^1(\mathbb{R}^3,\mathbb{C}^2)\times H^1(\mathbb{R}^3,\mathbb{C}^2)\rightarrow \mathbb{R}$ as 
\[\mathcal{B}(h,w)=\Re\int_{\mathbb{R}^3} \nabla h\cdot \nabla w dx+\nu \Re\int_{\mathbb{R}^3} h\cdot w dx .\]
By direct computation, we obtain
\begin{align*}
   \langle \mathcal{J}'(f_{c_n}),f_{c_n}-h\rangle -\mathcal{B}(h, f_{c_n}-h) 
    =   \int_{\mathbb{R}^3} |\nabla (f_{c_n}-h)|^2 dx+\nu \Re\int_{\mathbb{R}^3} |f_{c_n}-h|^2 dx &\\
    -2m\Re\int_{\mathbb{R}^3}\left(P |u_{c_n}|^{s-2} f_{c_n}+\Gamma * (K|{u_{c_n}}|^\kappa) K |{u_{c_n}}|^{\kappa-2}{f_{c_n}} \right)\cdot (f_{c_n}-h) dx.&
\end{align*}
Together with assumptions on $P(x)$, $\Gamma(x)$, and $K(x)$, we have 
\[ \int_{\mathbb{R}^3} |\nabla (f_{c_n}-h)|^2 dx+\nu \Re\int_{\mathbb{R}^3} |f_{c_n}-h|^2 dx=o(1),\]
which implies
 $$
 \| f_{c_n}-h\|_{H^1}\to 0\,\,\, as \,\,\, n\to \infty.
 $$
 Therefore, we obtain
 \[\mathcal{J}(h)=e_0,\quad \mathcal{J}'(h)=0,\quad \text{and}\quad \|h\|_{L^2}=1.\]
Hence  $h$ is a (weak) solution of the following equation
\begin{equation*}
  \begin{cases}
   & -\Delta h +\nu h=  2mP|h|^{s-2}h + 2m \Gamma*(K |h|^\kappa)|h|^{\kappa-2}h, \\
&\|h\|_{L^2}=1.
  \end{cases}
\end{equation*}
{\bf Claim.}
  The function $h\in H^1(\mathbb{R}^3,\mathbb{C}^2)$ is a ground state of \eqref{NSE}, that is $$\mathcal{E}(h)  = \inf\limits_{\|g\|_{L^2} = 1} \mathcal{E}(g).$$
\begin{proof}
  It is sufficient to prove 
  $$
   \limsup\limits_{n\to\infty} \, m( \mathcal{I}^{c_n}({u_{c_n}}) - mc_n^2)
    \leq \inf\limits_{\|g\|_{L^2} = 1} \mathcal{E}(g).
  $$
  For $v \in H^1(\mathbb{R}^3, \mathbb{C}^2)$, $\|v\|_{L^2} = 1$, set $w_{c_n}=P^+_{c_n}\left(\begin{array}{l}v \\ 0\end{array}\right)$,
   Proposition \ref{prop:3.2} implies  there exists unique $\varphi  (w_{c_n}) \in   S_{W_{c_n}}$, such that
  $$
  \mathcal{I}^{c_n}\left(\varphi  (w_{c_n})\right)=\sup _{{u} \in   S_{W_{c_n}}} \mathcal{I}^{c_n}({u}).
  $$
It is clear that 
\begin{equation*}
  \begin{aligned}
    \left\|P^-_{c_n}\left(\begin{array}{l}v \\ 0\end{array}\right)\right\|_{L^2}^2&=\frac{1}{4}  
    \left\|\mathbf{U}^{-1}(\xi)(I_4 -\beta)\mathbf{U}(\xi)\left(\begin{array}{l}\hat{v}  \\ 0\end{array}\right)\right\|_{L^2}^2\\
    &=\frac{1}{4}\left\|\left(1-\frac{m{c_n}^2}{\lambda(\xi)}\right)\hat{v}(\xi)\right\|_{L^2}^2 + \frac{1}{4}\left\| \frac{c}{\lambda(\xi)} \xi \cdot \sigma\hat{v}(\xi) \right\|_{L^2}^2 \\
   & = o(1), \,\,\, as \,\,\,n\to  \infty.
  \end{aligned}
\end{equation*}
Similarly, we have 
$$
\left\|\nabla P^-_{c_n}\left(\begin{array}{l}v \\ 0\end{array}\right)\right\|_{L^2}^2 = o(1),
$$
which yields
$$
\left\|P^-_{c_n}\left(\begin{array}{l}v \\ 0\end{array}\right)\right\|_{H^1}^2 = o(1).$$
%then
%$$
%\left\|\left(\begin{array}{l}v \\ 0\end{array}\right) - \Lambda^+_c\left(\begin{array}{l}v \\ 0\end{array}\right)\right\|_{H^1}^2 = o(1)$$
%%这里是因为取得极大的那个函数的正部，可以写成k\phi， phi是正空间中L2膜长是1的元素， yinli3.2说明k极限是1
%%上式说明投影来的元素的phi展开系数趋紧于1，
Consequently, by Proposition \ref{prop:3.2}, we obtain
$$
\left\|\left(\begin{array}{l}v \\ 0\end{array}\right) - \varphi(w_{c_n})^+ \right\|_{H^1}^2 =\left\|w_{c_n} - \varphi(w_{c_n})^+ +P^-_{c_n}\left(\begin{array}{l}v \\ 0\end{array}\right)\right\|_{H^1}^2= o(1).$$
Notice
$
\mathscr{D}_{c_n} \varphi(w_{c_n})^+ = \sqrt{-c_n^2 \Delta+m^2 c_n^4}\varphi(w_{c_n})^+.
$
%and
%$$
%\sqrt{-c^2 \Delta+m^2 c^4}\leqslant mc^2 -\frac{\Delta}{2m},
%$$
Then, we obtain
$$
(\mathscr{D}_{c_n} \varphi(w_{c_n})^+, \varphi(w_{c_n})^+ )_{L^2}\leq  mc_n^2 +\frac{1}{2m}\| \nabla \varphi(w_{c_n})^+  \|_{L^2}^2\leq mc_n^2 + \frac{1}{2m}\| \nabla v  \|_{L^2}^2 + o(1).
$$
Therefore, 
$$
\mathcal{I}^{c_n}\left(\varphi  (w_{c_n})\right)\leq  \frac{1}{m} \mathcal{E}(v) +mc_n^2 + o(1).
$$
This implies
$$
\limsup\limits_{n\to\infty} \, m( \mathcal{I}^{c_n}({u_{c_n}}) - mc_n^2)
 \leq \inf\limits_{\|g\|_{L^2} = 1} \mathcal{E}(g).
$$
Then by Lemma \ref{lem443}, we have 
$$
\mathcal{E}(h)= \inf\limits_{\|g\|_{L^2} = 1} \mathcal{E}(g).
$$
This ends the proof of this claim.
\end{proof}
\medskip

Therefore, we have $h$ is a ground state solution of the nonlinear Schr\"odinger equations \eqref{NSE}. 
This ends  the proof of Theorem 1.2.

\medskip

\noindent [\textbf{Proof of Theorem 1.3}]
First, we need to verify  $\{u_c\}$ is bounded in  $L^{\infty}$.
Lemma \ref{lemm:4.1} yields  there exist a constant $\mathcal{C}$, such that
  $$
  \| u_c \|_{L^p}\leq  \mathcal{C}  ,\quad \forall p\in [2,6].
  $$
  Dividing (6) by $c$, and combining Lemma \ref{lemm:2.4} and Lemma \ref{lemm:4.2}, we have 
  \begin{equation}\tag{4.6}\label{10}
    \begin{aligned}
      \|\nabla g_c\|_{L^3}^3&\lesssim \frac{(mc^2+\omega_c)^3}{c^3}\| f_c\|_{L^3}^3+\frac{1}{c^3}\int_{\mathbb{R}^3}|u_c|^{3s-6}|f_c|^3+\frac{1}{c^3}\int_{\mathbb{R}^3}\left( \Gamma * K|{u_c}|^\kappa K|{u_c}|^{\kappa-2}|{f_c}|\right)^3\\
     & \lesssim c^3  \| f_c\|_{L^3}^3+\frac{1}{c^3}\| u_c\|_{L^{3s-3}}^{3s-3}+\frac{1}{c^3}\| u_c\|_{L^{3\kappa-3}}^{3\kappa-3}\lesssim \mathcal{C} .
    \end{aligned}
  \end{equation}
Similarly, we get
\begin{equation}\tag{4.7}\label{11}
  \|\nabla f_c\|_{L^3}^3\lesssim \frac{1}{c^3}  \|  g_c\|_{L^3}^3+\frac{1}{c^3}\| u_c\|_{L^{3s-3}}^{3s-3}+\frac{1}{c^3}\| u_c\|_{L^{3\kappa-3}}^{3\kappa-3}\lesssim \mathcal{C}.
\end{equation}
Inequalities \eqref{10} and \eqref{11} yields
$$
  \|u_c\|_{W^{1,3}}\lesssim \mathcal{C}.
  $$
  Then   for any $r> 3$, we obtain
  $$
  \|u_c\|_{L^r}\lesssim \mathcal{C} _r .
  $$
  Since the  embedding $W^{1,3}(\mathbb{R}^3,\mathbb{C}^4)\hookrightarrow L^r(\mathbb{R}^3,\mathbb{C}^4)$ is continuous,
consequently, we get
$$\|u_c\|_{W^{1,r}}\lesssim \mathcal{C} _r .$$
Using the Sobolev embedding theorem again, we have  $$\|u_c\|_{L^\infty}\leq   \mathcal{C} .$$
Set $$  \quad M= - \Gamma * ( K| u_c |^\kappa) K| u_c |^{\kappa-2}, \quad N= - P  | u_c |^{s-2}.$$
then
$$
 u_c(x)=-(\mathscr{D}_c-\omega_c)^{-1}(M+N) u_c =-\int_{\mathbb{R}^3}Q(x-y)(M+N) u_c (y)dy.
$$
The next thing to do is to compute $Q$.
For $ u_c =( u _1, u _2, u _3, u _4)\in H^{1}(\mathbb{R}^3,\mathbb{C}^4)$, we have $$\mathcal{F}(\mathscr{D}_c-\omega_c) u_c(x)=A\mathcal{F} u_c (x),$$ where
$$
A=c\sum_{k=1}^3\alpha_k x_k+mc^2 \beta -\omega I.
$$
Set $$B= \frac{1}{\omega_c}(c\sum_{k=1}^3\alpha_k x_k+mc^2 \beta),\quad 
B^{-1}= \frac{\omega_c^2}{c^2|x|^2+m^2c^4}B,
$$
then we have $$A =\omega_c B(I-B^{-1}).$$
Denote $T=\frac{\omega_c^2}{c^2|x|^2+m^2c^4}<1$,
thus
$$
A^{-1}=\frac{1}{\omega_c}B^{-1}\left( \sum_{k=0}^{\infty} B^{-k}\right)=\frac{1}{\omega_c}\frac{T}{1-T}(I+B),
$$
and
   $$
   \begin{aligned}
    &\mathcal{F}^{-1}(A ^{-1})=\mathcal{F}^{-1}\frac{|\omega_c|}{c^2|x|^2+m^2c^4-\omega_c^2}(I+B)\\
    &=\left(ic  \frac{\boldsymbol{\alpha} \cdot x}{|x|^2}+ i\sqrt{m^2c^4-\omega_c^2}\frac{\boldsymbol{\alpha} \cdot {x}}{|{x}|}+\beta mc^2 +\omega_c\right) \frac{1}{4 \pi c^2|{x}|e^{\sqrt[]{m^2c^2-\omega_c^2/c^2}|x|}}.
   \end{aligned}
   $$
Hence$$ Q(x)=\left(ic  \frac{\boldsymbol{\alpha} \cdot x}{|x|^2}+ i\sqrt{m^2c^4-\omega_c^2}\frac{\boldsymbol{\alpha} \cdot {x}}{|{x}|}+\beta mc^2 +\omega_c\right) \frac{1}{4 \pi c^2|{x}|e^{\sqrt[]{m^2c^2-\omega_c^2/c^2}|x|}}$$
Notice
$$
\sqrt[]{m^2c^2-\omega_c^2/c^2}\to \sqrt{\nu }
$$
and $M+N$ is bounded uniformly, then a direct application of an argument in \cite[Theorem 2.1]{MR1785381}, we  get  for any  $\delta$ which satisfies $0<\delta < \sqrt{\nu }$, there exist $0< C(\delta)<\infty $ such that for large $c$
$$
|u_c(x)|\leqslant Ce^{-\delta |x|}
$$
for all $x\in\mathbb{R}^3$. This ends the proof of Theorem 1.3.

\section*{Appendix}
\textbf{\bf Lemma \ref{lemm:2.6}}
Let $P\in L^{\infty}(\mathbb{R}^3)$, then for any ${u}  \in  H^{1/2}(\mathbb{R}^3,\mathbb{C}^4) $, with $\|{u}\|_{L^2}=1$,  ${u}=t w+u^- $, where $w=\frac{{u}_+}{\|{u}_+\|_{L^2}}$, $t=\sqrt{1-\|u^-\|_{L^2}^2}$, the following inequality holds
$$
\int_{\mathbb{R}^3 }P |{u}|^s dx\geq C  \int_{\mathbb{R}^3 }P |w|^sdx-C   \|u^- \|_{L^2}^2\|w\|_{H^{1/2}}^2- C  \|u^- \|_{H^{1/2}}^2.
$$
For $w=\mathbf{U}_{\mathrm{FW}}^{-1}\left(\begin{array}{l}v \\ 0\end{array}\right)$ with $v \in H^1\left(\mathbb{R}^3, \mathbb{C}^2\right)$ and $\|v\|_{L^2}^2=1$, we have

$$
\int_{\mathbb{R}^3 } P |{u}|^sdx\geq  C \int_{\mathbb{R}^3 } P |v|^sdx- C \|\nabla v\|_{L^2}^s- C  \|u^- \|_{L^2}^2\|w\|_{H^{1/2}}^2- C \|u^- \|_{H^{1/2}}^2.
$$

\begin{proof}
 Notice for $u=tw+u^-$, one has
 $$
 \begin{aligned}
  \int_{\mathbb{R}^3}P|u|^sdx\geq  & 2^{1-s}t^s\int_{\mathbb{R}^3}P|\omega|^s dx- \int_{\mathbb{R}^3}P|u^-|^sdx\\
\geq  &  C \int_{\mathbb{R}^3}P|\omega|^s dx-  C \|u^-\|_{L^2}^2\int_{\mathbb{R}^3}P|\omega|^sdx- \int_{\mathbb{R}^3}P|u^-|^sdx\\
\geq &  C \int_{\mathbb{R}^3 }P |w|^s dx- C  \|u^- \|_{L^2}^2\|w\|_{H^{1/2}}^2-  C \|u^- \|_{H^{1/2}}^2.
 \end{aligned}
 $$
Now  for $w=\mathbf{U}_{\mathrm{FW}}^{-1}\left(\begin{array}{l}v \\ 0\end{array}\right)$ with $v \in H^1\left(\mathbb{R}^3, \mathbb{C}^2\right)$ and $\|v\|_{L^2}^2=1$,
we have
$$
w(x)=\mathcal{F}^{-1}\mathbf{U}^{-1}\left(\begin{array}{l}\hat{v} \\ 0\end{array}\right)=\left(\begin{array}{l}\mathcal{F}^{-1}[\Upsilon _-(\xi)\frac{\boldsymbol{\sigma}\cdot\xi}{|\xi|}\hat{v}] \\ \mathcal{F}^{-1}[\Upsilon _+(\xi)\hat{v}]\end{array}\right)(x).
$$
Set $g(x)= \mathcal{F}^{-1}[\Upsilon _+(\xi)\hat{v}](x)$, $f(x)=\mathcal{F}^{-1}[1-\Upsilon _+(\xi)\hat{v}](x)$, then $g=v-f$ and
$$
\int_{\mathbb{R}^3}P|\omega|^s dx\geq \int_{\mathbb{R}^3}P|g|^sdx \geq  2^{1-s}\int_{\mathbb{R}^3}P|v|^sdx - \int_{\mathbb{R}^3}P|f|^sdx
.$$
 Since
 $$
 | 1-\Upsilon _+(\xi)|=\frac{1-\Upsilon _+(\xi)^2}{1+\Upsilon _+(\xi)}=\frac{\Upsilon_-(\xi)^2}{1+\Upsilon _+(\xi)}\leq \Upsilon_-(\xi)^2,
 $$
then we obtain
$$
\begin{aligned}
  \|f\|^2=\|(\lambda(\xi))^{1/2}| 1-\Upsilon _+(\xi)\hat{v}|\|_{L^2}^2\leq \|(\lambda(\xi))^{1/2}| \Upsilon _-(\xi)\hat{v}|\|_{L^2}^2= \frac{1}{2}(\|v\|^2-mc^2\|v\|_{L^2}^2)\leq \frac{1}{4m}\|\nabla v\|_{L^2}^2
,\end{aligned}
$$
which implies

$$
\int_{\mathbb{R}^3 } P |{u}|^sdx\geq C \int_{\mathbb{R}^3 } P |v|^sdx- C \|\nabla v\|_{L^2}^s- C \|u^- \|_{L^2}^2\|w\|_{H^{1/2}}^2-  C \|u^- \|_{H^{1/2}}^2.
$$
This ends the proof.
\end{proof}

\medskip
\noindent\textbf{\bf Lemma \ref{lemm:2.7}}
Let   $\Gamma   $  satisfies $(\Gamma_1),$ then
  $\mathscr{F}: H^{1/2}\to L^p$ is locally Lipschitz, where $p\in [2,3]$.

  \begin{proof}
    It is sufficient to show the case when $p=2$ and $3$. When $p=3$,
    for $ u _1, u _2\in H^{1/2}$, one has
  \begin{align*}
      \mathscr{F}( u _1)-\mathscr{F}( u _2)=\frac{1}{2}\Gamma*\left(K \left(| u _1|^\kappa-| u _2|^\kappa\right)\right)\left(K | u _1|^{\kappa-1}+K| u _2|^{\kappa-1}\right)&\\+\frac{1}{2}\Gamma*\left(K \left(| u _1|^\kappa+| u _2|^\kappa\right)\right)\left( K| u _1|^{\kappa-1}-K| u _2|^{\kappa-1}\right).&
  \end{align*}
    
Hence for $t_i$, $s_i>0$($i=1,2$) satisfying $\frac{1}{3}=\frac{1}{t_i}+\frac{1}{s_i},$ we have 
  $$
  \begin{aligned}
    \|\mathscr{F}( u _1)-\mathscr{F}( u _2)\|_{L^3}\lesssim   \|\Gamma*\left(| u _1|^\kappa-| u _2|^\kappa\right)\|_{L^{t_{1}}}\|| u _1|^{\kappa-1} +| u _2|^{\kappa-1}\|_{L^{s_{1}}}&\\
   +\|\Gamma*\left(| u _1|^\kappa+| u _2|^\kappa\right)\|_{L^{t_{2}}}\|| u _1|^{\kappa-1}-| u _2|^{\kappa-1}\|_{L^{s_{2}} }.&
  \end{aligned}
  $$
By  Young inequality, for $t_i$, $r_i$ satisfying $\frac{1}{t_i}+1= \frac{14-6\kappa}{6}+ \frac{1}{r_i}$, $i=1,2$, one has
$$
\|\Gamma*\left(| u _1|^\kappa-| u _2|^\kappa\right)\|_{L^{t_{1}}}\lesssim \|| u _1|^\kappa-| u _2|^\kappa\|_{L^{r_{1}}},
$$
$$
\|\Gamma*\left(| u _1|^\kappa+| u _2|^\kappa\right)\|_{L^{t_{2}}}\lesssim \left\|| u _1|^\kappa+| u _2|^\kappa\right\|_{L^{r_{2}}}.
$$
Notice
$$|| u _1|^\kappa-| u _2|^\kappa|\lesssim | u _1- u _2|\cdot \left|| u _1|^{\kappa-1}+| u _2|^{\kappa-1}\right|,$$
then for $r_1$, $m$, $n$ satisfying  $\frac{1}{r_1}=\frac{1}{m}+\frac{1}{n}$, one has
$$
\|\Gamma*\left(| u _1|^\kappa-| u _2|^\kappa\right)\|_{L^{t_{1}}}\|| u _1|^{\kappa-1}+| u _2|^{\kappa-1}\|_{L^{s_{1}}}\lesssim\| u _1- u _2\|_{L^m}\|| u _1|^{\kappa-1}+| u _2|^{\kappa-1}\|_{L^{s_{1}}}\|| u _1|^{\kappa-1}+| u _2|^{\kappa-1}\|_{L^n}
$$
where  $\frac{1}{s_1}+\frac{1}{m}+\frac{1}{n}=\kappa-1$. Let $s_1=n$, $m=\max\left\{2,\frac{3}{\kappa-1}\right\}$, then $s_1=n\in \left[\frac{6}{3\kappa-4},\frac{3}{\kappa-1}\right]$.
Hence
$$
\|\Gamma*\left(| u _1|^\kappa-| u _2|^\kappa\right)\|_{L^{t_{1}}}\|| u _1|^{\kappa-1}+| u _2|^{\kappa-1}\|_{L^{s_{1}}}\lesssim \left( \| u _1\|+\| u _2\| \right)^{2\kappa-2}\| u _1- u _2\|.
$$
Similarly, for $s_2$, $m$, $n$ satisfying $\frac{1}{s_2}=\frac{1}{m}+\frac{1}{n}$, one has
$$
\|\Gamma*\left(| u _1|^\kappa+| u _2|^\kappa\right)\|_{L^{t_{2}}}\|| u _1|^{\kappa-1}-| u _2|^{\kappa-1}\|_{L^{s_{2}} }\lesssim \| u _1- u _2\|_{L^m}\|| u _1|^\kappa+| u _2|^\kappa\|_{L^{r_{2}}}\|| u _1|^{\kappa-2}+| u _2|^{\kappa-2}\|_{L^n}
$$
where $\frac{1}{r_2}+\frac{1}{m}+\frac{1}{n}=\kappa-1$. Take $m=3$, $n =\frac{3}{\kappa-2}$, $r_2=\frac{3}{2\kappa-2}$, we get
$$
\|\Gamma*\left(| u _1|^\kappa+| u _2|^\kappa\right)\|_{L^{t_{2}}}\|| u _1|^{\kappa-1}-| u _2|^{\kappa-1}\|_{L^{s_{2}} }\lesssim  \left( \| u _1\|+\| u _2\| \right)^{2\kappa-2}\| u _1- u _2\|.
$$
Consequently,
$$
\|\mathscr{F}( u _1)-\mathscr{F}( u _2)\|_{L^3}\lesssim \left( \| u _1\|+\| u _2\| \right)^{2\kappa-2}\| u _1- u _2\|.
$$
When $p=2$, we can also get
$$
\|\mathscr{F}( u _1)-\mathscr{F}( u _2)\|_{L^2}\lesssim \left( \| u _1\|+\| u _2\| \right)^{2\kappa-2}\| u _1- u _2\|.
$$
Hence by interpolation inequality, for $p\in [2,3]$,
$$
\|\mathscr{F}( u _1)-\mathscr{F}( u _2)\|_{L^p}\lesssim \left( \| u _1\|+\| u _2\| \right)^{2\kappa-2}\| u _1- u _2\|.
$$
This ends the proof.
\end{proof}
\medskip

\noindent \textbf{Future remark.}
Although the method of nonrelativistic limit has effectively addressed some of the difficulties in the normalized solution problem of the nonlinear Dirac equations, there still exist some urgent public issues that need to be resolved, such as: (1) The case where the potential function does not have a compactness assumption, (2) The super-mass-critical case of the nonlinear term, i.e., $s\in (8/3,3)$, (3) How fast the nonrelativistic limit process is.

\section*{Acknowledge}
	This work was partially supported by the National Key R\&D Program of China (Grant No. 2022YFA1005601) and the National Natural Science Foundation of China (Grant No. 12201625 and No. 12271508).

\medskip 
My manuscript has no associated data.

\newpage
\bibliographystyle{plain}
\bibliography{Dirac}

\newpage 
\noindent {Qi Guo\\
School of Mathematics, Renmin University of China, Beijing, 100872, P.R. China\\
e-mail:qguo@ruc.edu.cn}
\medskip
\\
{Pan Chen\\
Institute of Mathematics, Academy of Mathematics and Systems Science\\
University of Chinese Academy of Sciences, Beijing 100190,  P.R.China\\
e-mail:chenpan2020@amss.ac.cn}
\medskip 
\\
{Yanheng Ding\\
Institute of Mathematics, Academy of Mathematics and Systems Science\\
University of Chinese Academy of Sciences, Beijing 100190,  P.R.China\\
e-mail:dingyh@math.ac.cn}
\medskip
\\
{Hua-Yang Wang\\
Institute of Mathematics, Academy of Mathematics and Systems Science\\
University of Chinese Academy of Sciences, Beijing 100190,  P.R.China\\
e-mail:wanghuayang@amss.ac.cn}

\end{document}